\documentclass[twoside,10pt]{amsart}
\date{\today}
\usepackage{enumerate}
\usepackage{graphicx}
\usepackage{amsmath}
\usepackage{verbatim}  
\usepackage{amssymb}
\usepackage{epsfig}
\usepackage{amsthm}
\usepackage{amscd}
\usepackage[usenames]{color}

\usepackage{hyperref} 

\definecolor{r}{rgb}{.6,0,.3}

\newcounter{intro}

\newtheorem{theo}[intro]{Theorem}

\newcommand{\Tan}{\operatorname{Tan}}

\usepackage[latin1]{inputenc}

\newcommand{\B}{\mathbb{B}}
\newcommand{\R}{\mathbb{R}}
\newcommand{\RR}{\mathbb{R}}
\newcommand{\rth}{\mathbb{R}^3}
\newcommand{\n}{\mathbb{N}}
\newcommand{\h}{\mathbb{H}}
\newcommand{\esf}{\mathbb{S}}

\newcommand{\diam}{\operatorname{diam}}

\newcommand{\dist}{\operatorname{dist}}

\def\rth{{\R^3}}
 \newcommand{\FF}{\mathcal{F}}
\newcommand{\df}{ \stackrel{\rm def}{=}}

\newcommand{\Int}{\operatorname{Int}}

\numberwithin{equation}{section}
\newtheorem{theorem}{Theorem}[section]
\newtheorem{corollary}[theorem]{Corollary}
\newtheorem{lemma}[theorem]{Lemma}
\newtheorem{claim}[theorem]{Claim}
\newtheorem*{claim*}{Claim}
\newtheorem{proposition}[theorem]{Proposition}
\newtheorem*{maintheorem*}{Main Theorem}
\newtheorem*{theorem*}{Theorem}
\newtheorem*{mainlemma*}{Main Lemma}
\newtheorem*{keylemma*}{Key Lemma}

\theoremstyle{definition}
\newtheorem{remark}[theorem]{Remark}
\newtheorem*{remark*}{Remark}
\newtheorem{definition}[theorem]{Definition}

\newtheorem*{conjecture}{Conjecture}

\title[Area minimizing surfaces in $\h^3$]{%
  Properly embedded, area-minimizing surfaces in hyperbolic $3$-space
}%
\author[F. Martín]{Francisco Mart\'\i{}n}
\address[Martín]{
  Departamento de Geometr\'\i{}a y Topolog\'\i{}a,
  Universidad de Granada,
  18071 Granada, Spain.
}
\email{fmartin@ugr.es}
\author[B. White]{Brian White}
\address[White]{%
Department of Mathematics,
Stanford University,
Building 380, Sloan Hall,
Stanford, California 94305, U.S.A.
}
\email{white@math.stanford.edu}
\thanks{
The authors would like to thank Christos Mantouldis for carefully
reading the manuscript and making many useful suggestions.
  The first author was partially supported by
  MCI-FEDER grants  MTM2007-61775 and MTM2011-22547.
  The second author was partially supported by  National Science Foundation
   grants~DMS-0406209 and DMS~1105330.
  }

\begin{document}

\maketitle

\begin{abstract}
We prove prove a bridge principle at infinity for area-minimizing surfaces in the hyperbolic space
 $\h^3$, and we use it to prove that any open, connected, orientable 
surface can be properly embedded in $\h^3$
 as an area-minimizing surface. 
 Moreover, the embedding can be constructed in such 
 a way that the limit sets of different ends are disjoint. 
\end{abstract}

\setcounter{tocdepth}{1}
\tableofcontents

\section{Introduction}
The construction of new examples of complete minimal surfaces in hyperbolic space has had a very powerful tool: the  solvability of the asymptotic Plateau problem. The asymptotic Plateau problem in hyperbolic space basically asks the existence of an area-minimizing submanifold in $ \h^{n +1} $ which is asymptotic to a given submanifold $ \Gamma^{n-1} \subset \partial\h^{n +1} $, where $ \partial\h^{n +1} $
represents the sphere of infinity of $ \h^{n +1} $, which we also call {\em the ideal boundary
of hyperbolic space.}

Using methods from  geometric measure theory, Michael Anderson
\cite{anderson1} solved the asymptotic Plateau problem for absolutely area-minimizing submanifolds in any dimension and codimension.

Anderson did not impose any restriction to the topology the solutions he 
{gets}, so we cannot get any
idea about their topological properties. In this way, it becomes
interesting (as in the classical Plateau problem) to find the area-minimizing solution 
but fixing {\em a priori} the topological type. In \cite{anderson2}, Anderson focused on the asymptotic Plateau problem with the type of a disk and provided an existence result in dimension 3.

Moreover, in \cite{anderson2}, Anderson built a special Jordan curves in 
$ \partial\h^3 $, such that the surface obtained as a solution to the  asymptotic Dirichlet problem cannot be a plane. In fact, he built examples of genus $ g> g_0 $ for a particular genus
$ g_0$.
In the same context, de Oliveira and Soret \cite{oliveira-soret}  demonstrated the existence of complete and stable minimal surfaces  in hyperbolic 3-space for any orientable {\em finite} topological
type\footnote{A surface has finite topological type 
 if it has the topology of a compact surface minus a
{\bf finite} number of points.}. They also studied the isotopy type of these surfaces in some special cases. 
The main difference with the result  of Anderson is that Anderson begins with asymptotic data, and gives an area-minimizing surface with that particular data but without any kind of 
control over the topological type, while Oliveira and Soret
start with a  surface with boundary and build a stable embedded minimal surface
in the hyperbolic space whose asymptotic (or ideal) boundary is determined essentially
by the surface. In this setting, we can frame the following conjecture:
\begin {conjecture} [A. Ros]
{Every} open, connected, orientable surface\footnote {We say that a connected surface is open if it is non-compact and has no boundary.}  can be properly and  minimally embedded in $ \h^3 $.
\end {conjecture}
This paper is devoted to give a positive answer to the problem above. To be more precise we prove:
\begin{theo}
Every open, connected, orientable surface can be properly embedded in $\h^3$ as an area-minimizing surface. Moreover, the embedding can be constructed in such a way that the limit sets of different ends are disjoint.
\end{theo}
The definition of  ``area-minimizing'' and ``uniquely area-minimizing'' surfaces can be found in 
Section~\ref{sec:area-minimizing}
(Definition~\ref{def:area}.) The fundamental tool in solving this problem has been the bridge principle at infinity (Section~\ref{sec:bridge}) which can be stated in these terms:
\begin{theo}[\bf Bridge principle at infinity]
 Let $S$ be an open, properly embedded, uniquely area-minimizing surface 
  in $\h^3$ 
 whose closure $\overline{S}\subset\overline{\h^3}$ is a smooth manifold-with-boundary.  
 Let $\Gamma$ be a smooth arc in $\partial\h^3$ {meeting
$\partial S$ orthogonally} and satisfying $\Gamma \cap \partial S=\partial \Gamma$.
   
  Consider a sequence of bridges $P_n$ 
 on $\partial\h^3$ that shrink nicely to $\Gamma$. If $S$ is {\bf strictly $L^\infty$ stable} (see Definition~\ref{def:linf}), then for all large enough $n$, there exists a strictly $L^\infty$ stable,  
 uniquely area-minimizing surface $S_n$ that is properly embedded in $\h^3$ such that:
 \begin{enumerate}[\rm 1)]
 \item $\overline{S_n}$ is a smooth, embedded manifold-with-boundary in $\overline{\h^3}$.
 \item $\partial S_n = (\partial  S \setminus \partial P_n) \cup ( \partial P_n \setminus \partial  S)$.
  \item The sequence $\overline{S_n}$  converges smoothly to       
   $\overline{S}$ on compact subsets of $\overline{\h^3}\setminus\Gamma$.
  \item The surface $\overline{S_n}$ is homeomorphic to $\overline{S} \cup P_n$.
  \end{enumerate} 
\end{theo}

This bridge principle gives us some flexibility in order to construct properly embedded area-minimizing surfaces in $\h^3$ with arbitrary infinite topology and some kind of regularity at infinity. 
\begin{theo}
If
 $S$ is  a connected, open, orientable surface with infinite topology, then there exists a proper, area-minimizing
 embedding of
$S$ into $\h^3$ such $\overline{S}$ is a smooth embedded manifold-with-boundary
except at a single point of $\partial S\subset \partial\h^3$.
\end{theo}
Finally, we would like to point out that the same methods allow us to construct properly embedded 
area-minimizing surfaces so that the limit set is the whole ideal boundary $\partial\h^3$.
\vskip 5mm

\section{Preliminaries}
Throughout
 this paper $\h^{n+1}$ will represent the (n+1)-dimensional hyperbolic space. 
We will use the models:
\begin{enumerate}
\item {\bf Poincaré's ball model:}  the open unit ball $\B^{n+1}$ of $\R^{n+1}$ endowed with Poincaré's metric $\displaystyle ds^2:= 4  \frac{\sum_{i=1}^{n+1} dx_i^2}{(1-\sum_{i=1}^{n+1} x_i^2)^2}$.
\item  {\bf Poincaré's half-space model:} the upper half-space $\{x_{n+1} >0 \} \subset \R^{n+1}$, endowed with the metric $\displaystyle ds^2:=  \frac{1}{x_{n+1}^2} \sum_{i=1}^{n+1} dx_i^2$.
\end{enumerate} Let $\overline{\h}^{n+1}$ denote the usual compactification of $\h^{n+1}$. As we 
mentioned in the introduction, we shall denote the ideal boundary as 
$\partial\h^{n+1}:= \overline{\h}^{n+1} \setminus \h^{n+1}$. 
Observe that $\partial\h^{n+1}$ is diffeomorphic to the sphere $\esf^n$.
(In the ball model, it is $\partial \B^{n+1}$ and in the upper half space model
it is $\{x: x_{n+1}=0\}\cup \{\infty\}$.)

\subsection{Simple exhaustions}
 One of the main tools in the proofs of the theorems
stated in the introduction is the existence of a particular kind of exhaustion for any open
surface. 
In \cite{fmm}, Ferrer, Meeks and the first author  proved that every
open,
 connected,  
 orientable  surface $M$ 
 has a {\em simple exhaustion}, i.e., a smooth, compact
 exhaustion $M_1\subset M_2\subset \cdots$ such that:
\begin{enumerate}
\item $M_1$ {is} a disk.
\item For all $n \in \n$, each component of $M_{n+1} \setminus \Int(M_n)$ has one boundary component in
$\partial M_n$ and at least one boundary component in $\partial M_{n+1}$.
\item\label{infinite-topology-item}
 If $M$ has infinite topology, then for all $n \in \n$, $M_{n+1} \setminus \Int(M_n)$ contains a unique nonannular component; that component is 
topologically a pair of pants or an annulus with a handle.
\item If $M$ has finite topology (with genus $g$ and $k$ ends),
property~\ref{infinite-topology-item} holds for $n\leq g+k$, and when $n> g+k$, all of the
components of $M_{n+1} \setminus \Int(M_n)$ are annular.
\end{enumerate}
\begin{figure}[htbp]
    \begin{center}
        \includegraphics[width=.95\textwidth]{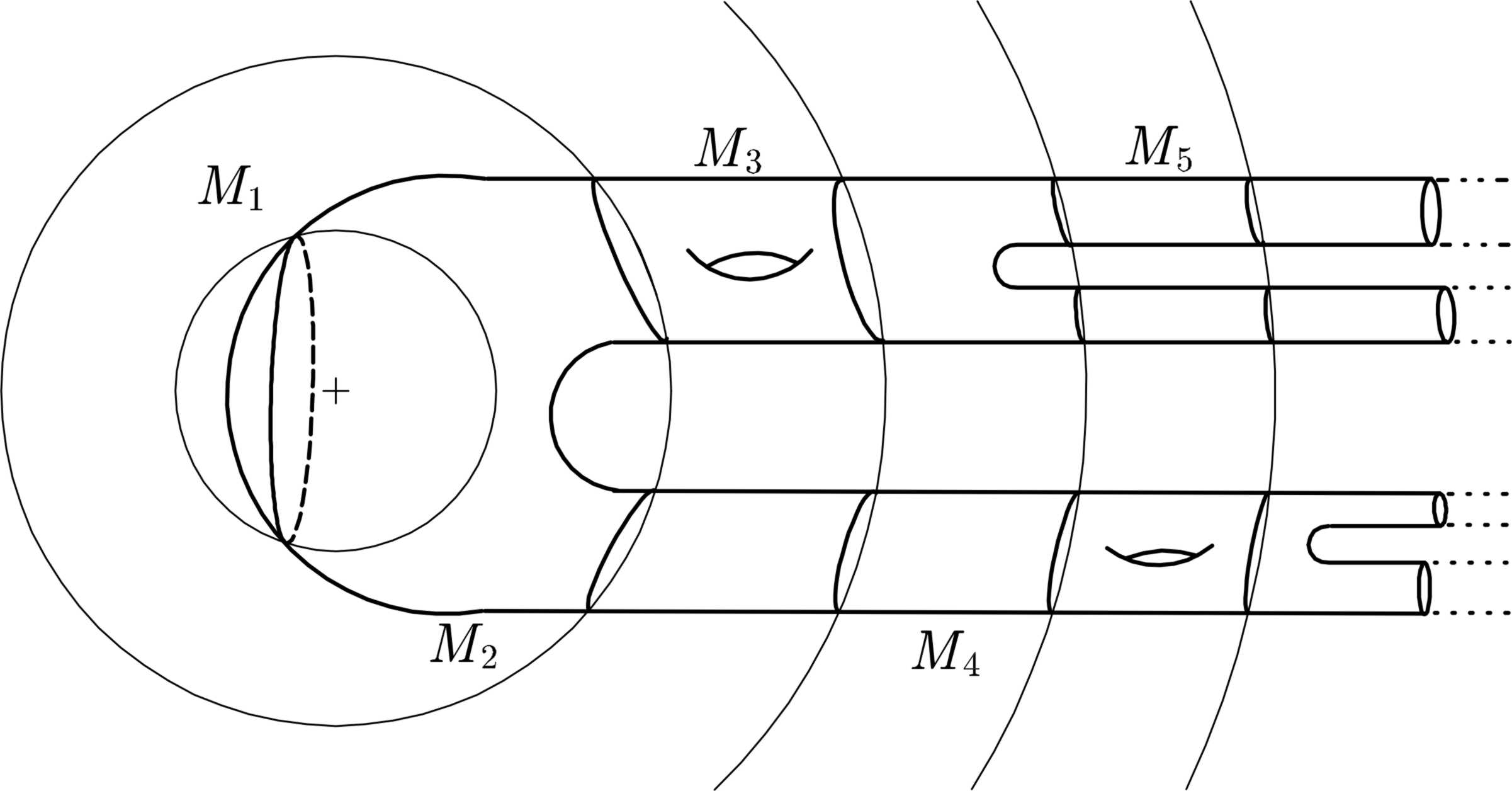}
           \end{center}
   \caption{A simple exhaustion of $M$.} \label{fig:simple}
\end{figure}

\begin{remark}\label{hare-remark}
For the purposes of this paper, one could replace~\ref{infinite-topology-item} in the definition of simple exhaustion
by the slightly weaker condition:
each component of $M_{n+1}\setminus \Int(M_n)$ is an annulus, a pair of pants, or an annulus
with a handle.  Which definition one uses does not affect any of the proofs.
\end{remark}

\subsection{Limit sets}
We are also interested in the asymptotic behavior of the minimal
surfaces we are going to construct. So, we need some background about
the limit set of an end.

\begin{definition}\label{def:limit}
  Let $\psi\colon S \to \h^3$ be an immersion of a
  surface $S$ with possibly non-empty boundary. The {\bf limit set} of
  $S$ is $L(S)=\bigcap_{\alpha \in I}\overline{\psi(S\setminus C_\alpha)},$ where
  $\{ C_{\alpha}\}_{\alpha \in I}$ is the collection of compact subdomains of
  $S$ and the closure $\overline{\psi(S\setminus C_\alpha)}$ is taken in
  $\overline{\h^3}$. The {\bf limit set $L(E)$ of an
    end $E$ of $S$} is defined to be the intersection of the limit
  sets of all properly embedded subdomains of $S$ with compact
  boundary which represent $E$. 
\end{definition}

Note that $L(S)$ is a closed set, and that if $E$ is an end of $S$, then $L(E)$ is a
closed subset of $L(S)$.
Note also that $\psi:S\to \h^3$ is proper if and only if $L(S)\subset \partial \h^3$.
(Recall that a continuous map $f:X \to Y$ between topological spaces
is called {\em proper} provided the inverse image of every compact 
set is compact.)
Thus if $S$ is an open, proper submanifold of $\h^3$, then $L(S)$
is the equal to the set theoretic boundary $\partial S=\overline{S}\setminus S$.
More generally, if $S$ is an open, proper submanifold of an open subset $U$ of $\h^3$,
then $L(S) = \partial S = \overline{S}\setminus S$, where $\overline{S}$ denotes the
closure of $S$ in $\overline{\h^3}$.

\begin{proposition}[Convex Hull Property]\label{convex-hull-proposition}
Let $M$ be an open minimal surface  in $\h^3$.
Then $M$ is contained in the convex hull of its limit set.
In other words, if $\Sigma\subset \h^3$ is a totally geodesic plane
and if $L(M)$ lies in the closure $N$ of one component of $\h^3\setminus \Sigma$,
then $M$ also lies in $N$.
\end{proposition}

Recall that, in general,  $L(M)$ may include points in $\h^3$ and also points in $\partial \h^3$.

\begin{proof}
We can assume in the upper half space model that
\[
  N = \{ (x,y,z): \text{$z\ge 0$ and $x^2+y^2 + z^2 \ge 1$} \}.
\]
The level sets of the function $(x,y,z)\mapsto x^2+y^2+z^2$
are minimal surfaces, so by the strong maximum principle,
its restriction to $\overline{M}$ cannot attain its minimum at
a point of $M$.
\end{proof}

\begin{theorem}[Strong local uniqueness theorem]\label{strong-local-uniqueness-theorem}
Let  $M$ be an open minimal
surface in $\h^3$ and $\Gamma$ be a curve in $\partial\h^3$ such
that $M'=M\cup\Gamma$ is a smooth, embedded submanifold (with boundary) of $\overline{\h^3}$.
Then each point $p\in \Gamma$ has a neighborhood $U\subset \overline{\h^3}$
with the following property: if $S\subset \h^3$ is an open minimal surface with
$L(S)\subset M'\cap U$, then $S\subset M\cap U$.
\end{theorem}

\begin{proof}
We use the upper halfspace model.
Let $v$ be a vector normal to $\overline{M}$ at $p$.  Note that $v$ is horizontal.
(One easily shows, using totally geodesic barriers, that
$M'$ meets $\partial \h^3$ orthogonally. See, for example, \cite{hl}.)
We may assume that each line in $\RR^3$ parallel to $v$ intersects $M'$ at most
once.  (Otherwise replace $M'$ by $M'\cap \B(p,R)$ with $R$ sufficiently small.)
  It follows that $M'$ and its translates by multiples of $v$ foliate an open subset $W$
of $\overline{\h^3}$.  Choose $r>0$ small enough that $\B(p,r)\cap \overline{\h^3}$ is contained in $W$.
Then $U:=\B(p,r)$ has the desired property.
For suppose that $S$ is a minimal surface in $\h^3$ with $L(S)\subset M'\cap U$.
By the convex hull property (Propostion~\ref{convex-hull-proposition}), $S\subset U$.
The strong maximum principle then forces $S$ to lie in $M$.
(Consider the maximum value of $|t|$ such that $S$ intersects $M'+ tv$.)
\end{proof}

\section{Area-minimizing surfaces}\label{sec:area-minimizing}

In this section, we present some fundamental theorems about area-minimizing surfaces
in hyperbolic space.  Those theorems will be used repeatedly in the rest of the paper.

\begin{definition} \label{def:area}
Suppose $S \subset \h^3$ is a (possibly nonorientable) 
compact surface with unoriented boundary. 
The surface $S$ is called
{\bf area-minimizing}\footnote{In the literature, ``area-minimizing" as defined here is usually referred
to as ``area-minimizing mod $2$".}
 if $S$ has least area among all surfaces (orientable or nonorientable)
with the same boundary. For a noncompact surface $S$, we say that $S$ is 
area-minimizing provided each compact portion of it is area-minimizing.

Now suppose that $S$ is an open, properly embedded, area-minimizing surface
in $\h^3$.
  We say that $S$ is {\bf uniquely area-minimizing}
if it is the only area-minimizing surface with boundary $\partial S$.
\end{definition}

For example, the convex hull property (Proposition~\ref{convex-hull-proposition})
implies that a  totally geodesic plane is uniquely area-minimizing.

\begin{theorem}[Boundary Regularity Theorem]\label{boundary-regularity-theorem}
Let $M$ be an open, area-minimizing surface in $\h^3$.
Suppose that $W$ is an open subset of $\overline{\h^3}$  with the following property:
\[
  L(M) \cap W \subset \Gamma
\]
where $\Gamma$ is a smooth, connected, properly embedded curve in $W\cap \partial \h^3$.
Then either $L(M)\cap W$ is empty, or $L(M)\cap W=\Gamma$ and $M\cup\Gamma$ is 
a smooth manifold-with-boundary.
\end{theorem}

\begin{proof}
Hardt-Lin \cite{hl} prove that in a neighborhood $U$ of each point of $\Gamma$,
$\overline{M}\cap U$ is a union of some finite number $\kappa$ of $C^1$ manifolds-with-boundary, the boundary
being either $\Gamma\cap U$ or the empty set, 
and that those manifolds are disjoint except at the boundary.
Their result is stated for integral currents, but their proof also works for chains mod $2$ and
in that case actually gives more:  $\kappa$ must then be $0$ or $1$ (because in 
Lemma~2.1 of their paper, if $\delta$ is sufficiently small, then $\kappa$ must be $0$ or $1$.)
A priori the number $\kappa$ might depend on the point, but since it is locally constant
and since $\Gamma$ is connected, it must in fact be constant on $\Gamma$.
In case $\kappa=1$, Tonegawa \cite{tonegawa} improves the boundary regularity by showing
that $M\cup \Gamma$ is $C^\infty$.
\end{proof}

\begin{theorem}[Compactness Theorem]\label{compactness-theorem}
Let $M_i$ be a sequence of open, area-minimizing surfaces that are properly 
embedded in an open subset $U$ of $\h^3$.  Then (after passing to a subsequence)
the $M_i$ converge smoothly with multiplicity $1$
on compact subsets of $U$ to such a surface $M$.

Now suppose that $U= W\cap \h^3$, where $W$ is an open subset of $\overline{\h^3}$.
Suppose also that  $L(M_i)\cap W$ is a smooth embedded curve $\Gamma_i$ in 
  $W\cap \partial\h^3$,
and that  $\Gamma_i$ converges smoothly and with multiplicity $m$ to an embedded curve
$\Gamma$ in $W\cap\partial\h^3$. 

Then $M_i\cup \Gamma_i$ (which is a smooth manifold-with-boundary by 
Theorem~\ref{boundary-regularity-theorem})
converges (in the Hausdorff topology on the space of relatively closed
subsets of $W$) to $M\cup \Gamma$.   
Furthermore:
\begin{enumerate}[\rm (1)]
\item\label{m-odd-item} If $m$ is odd,  then $M \cup \Gamma$ is a smooth manifold-with-boundary.
\item\label{m-even-item} If $m$ is even, then $\overline{M}$ is disjoint from $\Gamma$.
\item\label{m-one-item}  If $m=1$, then  $M_i\cup\Gamma_i$ converges smoothly on compact
subsets of $W$ to $M\cup \Gamma$.  
\end{enumerate}
\end{theorem}

\begin{proof}
The statement about smooth convergence on compact
subsets of $U$ is very standard.  (The areas of the $M_i$ are uniformly locally bounded,
since if $\Omega$ is a bounded, open region in $\h^3$, then
\newcommand{\area}{\operatorname{area}}
\begin{equation}\label{half-bound}
  \area(M_i\cap \Omega) \le \frac12 \area(\partial\Omega)
\end{equation}
 Also, the curvatures of the $M_i$ are uniformly bounded on compact
subsets by standard curvature estimates for area-minimizing hypersurfaces 
of dimension $<8$.  That the multiplicity of $M$ is $1$ follows from~\eqref{half-bound}.)

The convergence  of $M_i\cup \Gamma_i$ to $M\cup \Gamma$  (as relatively closed subset of $W$)
 follows from 
the convergence of $M_i$ to $M$ in $U$, the convergence of $\Gamma_i$ to $\Gamma$ in $W$,
and the convex hull property.  (The convex hull property ensures that points $p_i\in M_i$
cannot converge subsequentially to a point in $W_\infty\setminus \Gamma$, where 
$W_\infty:=W\cap\partial \h^3$.)

Since the remaining assertions are local, we can assume (in the upper half space model)
that $W$ is $\overline{\h^3}\cap\B$ for some open Euclidean ball centered at a point in $\partial\h^3$.
We can also assume that $\Gamma$ is connected, so that it divides $W_\infty:=W\cap\partial \h^3$
into two components.  Let $p$ and $q$ be a pair of points lying in different components of
$W_\infty\setminus \Gamma$, and let $C$ be a smooth curve joining $p$ to $q$
such that $C\setminus \{p,q\}$ lies in $U$.
By perturbing $C$ slightly, we may assume it intersects $M$ transversely.
Let $\nu$ be the mod $2$ number of points of $M\cap C$; that number is independent
of $C$ (for $C$ transverse to $M$).  
 By the smooth convergence, $M_i$ intersects $C$ transversely for large $i$
and the mod $2$ number $\nu_i$ of intersection points is independent of $C$.
The smooth convergence $M_i\to M$ also implies that $\nu_i=\nu$ for all sufficiently large $i$.
By elementary topology that $\nu_i\cong m\, (\text{mod $2$})$ for $i$ sufficiently large.
(If this not clear, note that $C$ can be homotoped in $W$ to a curve in $W_\infty$ that intersects
$\overline{M_i}$ transversely in exactly $m$ points.)

Thus $m$ even implies that $\nu=0$ and $m$ odd implies that $\nu=1$.
Assertions~\eqref{m-odd-item} and~\eqref{m-even-item} now follow immediately from the 
Boundary Regularity Theorem~\ref{boundary-regularity-theorem}.
Assertion~\eqref{m-one-item} follows from the boundary regularity estimates of Hardt-Lin
and Tonegawa.  
\end{proof}

\newcommand{\eps}{\epsilon}
We remark that if $m>1$, then the convergence of $M_i\cup \Gamma_i$ to $M\cup\Gamma$
fails to be smooth along $\Gamma$.  For example, suppose in the ball model $\B=\h^3$
that $\Gamma_i$ is the union of the two circles $\partial \B \cap \{z=\pm \eps_i\}$
where $\eps_i\to 0$. 
Then $\Gamma_i$ converges smoothly with multiplicity $m=2$
to the equator $\Gamma:=\partial \B \cap \{z=0\}$.
 Let $M_i$ be an area-minimizing surface with boundary $\Gamma_i$.
(Such a surface exists by a theorem of Anderson -- see Theorem~\ref{th:principal} below.)
For small $\eps_i$, one can prove that $M_i$ is a minimal annulus that lies
within Euclidean distance $O(\eps_i)$ from $\Gamma$.
Thus $M_i\cup \Gamma_i$ converges to $M\cup\Gamma$, where $M$ is the empty set.
Note that the convergence of $M_i\cup\Gamma_i$ to $\Gamma$ is not smooth.

\begin{definition}
We say that a closed set $K\subset \partial\h^3$ has {\em piecewise smooth
boundary} provided 
\begin{enumerate}
\item $K$ is the closure of its interior, and
\item there is a finite set $S$ of points such that $(\partial K)\setminus S$ is the
disjoint union of a finite set of smooth curves.
\end{enumerate}
\end{definition}

\begin{theorem}[Basic Existence Theorem] \label{th:principal}
Let $K\subset  \partial\h^3$ be a closed region with piecewise smooth boundary.
Then there is an area-minimizing surface $M$ in $\h^3$ such that $\partial M=\partial K$.
Furthermore, if $M$ is any area-minimizing surface in $\h^3$ with $\partial M=\partial K$, then
\begin{enumerate}
\item $\overline{M}$ is a smooth embedded manifold with boundary except
at the finite set of points where $\partial K$ is not a smooth embedded curve.
\item there is a unique open subset $E(M,K)$ of $\h^3$ whose boundary in $\overline{\h^3}$
is $K\cup M$.
\end{enumerate}
\end{theorem}

 \begin{proof} 
 Anderson \cite[Theorem~3]{anderson1}
 proves existence of a smooth, area-minimizing surface $M\subset H$ with the property
that $\partial M=\partial K$
as flat chains mod $2$ with respect to the Euclidean metric on the ball. (He states the theorem
for integral currents, but exactly the same proof works for chains mod $2$.)  In particular, this
implies that $\overline{M}\setminus M = \partial K$ as sets 
(i.e., in the notation of~\ref{def:limit}, 
that $L(M)=\partial K$.)

The smoothness of $\overline{M}$ at the regular points of $\partial K$ follows immediately from 
the Boundary Regularity Theorem~\ref{boundary-regularity-theorem}.

If we identify $\h^3$ conformally with a ball $B$ in $\RR^3$,
then  $M\cup K$ 
becomes (except possibly at finitely many points) a compact, embedded, piecewise-smooth  
closed manifold of  $\RR^3$ contained in $\overline{B}$.
By elementary topology, there is a unique open subset $E(M,K)$ of $B$ 
whose boundary is $M\cup K$.
\end{proof}

\begin{lemma} \label{lem:U}
Let $M$ be an area-minimizing surface.  Let $M'$ be a compact region in the interior of $M$
such that $M'$ has piecewise smooth boundary.
Then $M'$ is the unique area-minimizing surface with its boundary.
\end{lemma}

\begin{proof}
Standard. 
\end{proof}

\begin{theorem}
Let $K_1$ and $K_2$ be disjoint, closed regions in $\partial\h^3$ with piecewise smooth boundaries.
Let $M_1$ and $M_2$ be least area surfaces with boundaries $\partial K_1$ and $\partial K_2$,
and let $U_i$ be the region enclosed by $M_i\cup K_i$.  Then $\overline{U_1}$ and $\overline{U_2}$
are disjoint.
\end{theorem}

\begin{proof}
Let $Z=\overline{U_1}\cap \overline{U_2}$.  Note that $Z$ is a compact subset of $\h^3$.
Suppose it is nonempty.   Then $U_1\cap U_2$ is nonempty by the maximum principle (applied
to $M_1$ and $M_2$.)  By Lemma~\ref{lem:U},  $U_1\cap M_2$ is the unique least area surface with its boundary.
Likewise, $U_2\cap M_1$ is the least area surface with its boundary.  But $U_1\cap M_2$ and $U_2\cap M_1$
have the same boundary, a contradiction.
\end{proof}

\begin{corollary}
Suppose for $i=1,2$ that $K_i$ is a closed region in $ \partial\h^3$ and that $M_i$ is a least area surface
in $\h^3$ with $\partial M_i=\partial K_i$.  Let $U_i$ be the region enclosed by $M_i\cup K_i$.
If $K_1$ is contained in the interior of $K_2$, then $U_1\cup M_1$ is contained in $U_2$.
\end{corollary}

(This corollary is not really a corollary -- but it is proved in exactly the same way as the theorem. Actually, we use the 
corollary but not the theorem.)

\begin{theorem} \label{th:E}
Let $K$ be a closed region in $ \partial\h^3$ with piecewise smooth boundary.
Let $\FF$ be the collection of all least area surfaces in $\h^3$ with boundary $\partial K$.
Then $\FF$ contains surfaces $M_{\rm in}$ and $M_{\rm out}$ with the following property.
If $M\in \FF$, then
\[
   E(M_{\rm in},K) \subset E(M,K) \subset E(M_{\rm out},K).
\]
\end{theorem}

Recall the $E(M,K)$ is the region enclosed by $M$ and $K$.
(We think of $M_{\rm in}$ and $M_{\rm out}$ as the innermost and outermost surfaces in the
family $\FF$.)

\begin{proof}
Let $K_1\subset K_2 \subset \dots$ be a sequence of closed subsets of the interior of $K$ such
that each $K_i$ has smooth boundary, such that $\cup K_i$ is the interior of $K$, such
that $\partial K_i\to \partial K$, and such that convergence $\partial K_i$ to $\partial K$ is smooth
except at the points where $\partial K$ is not smooth.

Let $M_i$ be a least area surface with boundary $\partial K_i$, and let $M_{\rm in}$ be a subsequential
limit of the $M_i$.  Then $M_{\rm in}\in \FF$.

Furthermore, if $M\in \FF$, then 
\[
    E(M_i,K_i)\subset E(M, K)
\]
for all $i$ (by the lemma), and thus $E(M_{\rm in},K)\subset E(M,K)$.   

The assertions about $M_{\rm out}$ are proved in a very analogous manner.
\end{proof}

\begin{remark} \label{re:unique}
Note that $M_{\rm in}$ is unique, as is $M_{\rm out}$.  Hence if $g$ is an isometry of $\h^3$ such
that $g(K)=K$, then $g(M_{\rm in})=M_{\rm in}$ and $g(M_{\rm out})=M_{\rm out}$.
\end{remark}

Of course $M_{\rm in}=M_{\rm out}$ if and only if there is only one least area surface with boundary $K$.

\section{Strict $L^\infty$ Stability}

In this section, we define strict $L^\infty$ stability 
 and we prove
some of its basic properties.
Let $\Omega$ be a Riemannian manifold that is connected but not compact.

\begin{definition}[strict $L^\infty$ stability] \label{def:linf}
Let $J$ be a self-adjoint 2nd-order linear elliptic operator on a surface $\Omega$.
Let us say $\Omega$ is {\bf strictly $L^\infty$ stable} (with respect to $J$) if the first eigenvalue
of any compact subdomain is strictly positive and if there are no nonzero bounded Jacobi fields (i.e.
solutions of $Ju=0$) on $\Omega$.
\end{definition}

Throughout this paper, we will use the concept of strict $L^\infty$ stability
only for minimal surfaces, and the operator $J$ will always be the Jacobi operator.
(However, the following three results hold for general manifolds $\Omega$ and operators
$J$.)

\begin{lemma}\label{standard}
Let $w$ be a positive solution of $J \, w=0$ on $\Omega$.
Then the first eigenvalue of $J$ on every compact subdomain of $\Omega$ is strictly positive.
\end{lemma}

The proof is standard.  See, for example, Theorem~1 of
\cite{Fischer-Colbrie-Schoen}.

\begin{lemma}\label{ratio-lemma}
Let $u$ and $w$ be Jacobi fields on a connected minimal hypersurface $M$.
Suppose that $u/w$ has a positive local maximum $\lambda$ at a point $p$ where  $u$ and $w$ are both positive.  Then $u=\lambda w$.
\end{lemma}

\begin{proof}
By hypothesis, $u-\lambda w$ has a local maximum value $0$.  Thus by the
strong maximum principle, $u-\lambda w$ vanishes in a neighborhood of $p$.
By the unique
continuation property for solutions of second order elliptic equations, 
$u-\lambda w\equiv 0$.
\end{proof}

\begin{theorem}\label{comparison}
Suppose $w$ is a positive solution of $Jw=0$ such that $\lim_{p\to \partial \Omega}w(p)=\infty$.
Then $\Omega$ is strictly $L^\infty$ stable.
\end{theorem}

\begin{proof}
We have to show that each compact subdomain is stable and that there are no nonzero bounded
Jacobi fields on $\Omega$.  
By Lemma~\ref{standard}, each compact subdomain is stable.
Thus we need only show that there are no nonzero, bounded Jacobi fields.

Suppose $u:\Omega\to \RR$ is a nonzero, bounded Jacobi field on $\Omega$.
We may suppose that $u>0$ at some points.
Since $u/w$ is positive at some points and tends to $0$ on $\partial \Omega$, 
 it has a  local maximum $\lambda>0$ at some point $\Omega$.  
 By Lemma~\ref{ratio-lemma}, $u\equiv kw$, which is impossible since $u$ is bounded $w$ is unbounded.
 \end{proof}

\begin{corollary} \label{co:plane}
A totally geodesic plane in $\h^3$ is strictly $L^\infty$ stable.
\end{corollary}

\begin{proof}
Without loss of generality we can assume that the plane is a hemisphere centered at the origin in the
 upper halfspace model of $\h^3$.
Consider the Jacobi field $w$ that comes from dilations about $0$.
\end{proof}

\begin{theorem}\label{th:boundary-point}
Let $M$ be an area-minimizing surface in $\h^3$ with $\partial M\subset \partial\h^3$.
Let $p$ be a regular point of $\partial M$, so that (in the upper halfspace model)
$M\cup\partial  M$ is a regular manifold-with-boundary near $p$.

Let $u$ be a bounded, nonnegative Jacobi field on $M$.  
Then $\lim_{q\to p}u(q)=0$.
\end{theorem}

\begin{proof}
Without loss of generality, $p=0$ in the upper half space model of $\h^3$.
Let $p_n\in M$ be points such that $p_n\to 0$ and such that 
\[
    u(p_n) \to \limsup_{q\to 0}u(q).
\]
Suppose the supremum limit is nonzero. Then we may assume it is $1$.
Now make a Euclidean translation and dilation of $\h^3$ that moves
$M$ to $M_n$ and that moves $p_n$ to $(0,0,1)$.
Let $u_n$ be the Jacobi field on $M_n$ corresponding to $u$ on $M$.
After passing to a subsequence, the $M_n$ converge to a totally geodesic
plane $M^*$ and the $u_n$ converge to a bounded Jacobi field $u^*$ on $M^*$
that attains its maximum value ($1$) at the point $(0,0,1)$.  But that contradicts
the strict $L^\infty$ stability of a totally geodesic plane.
\end{proof}

\section{Minimal Strips and Skillets}

In this section we define and analyze minimal strips and minimal skillets.
They will be important for us because they arise as blowups in the proof
of the Bridge Theorem~\ref{th:bridge}.

\begin{theorem} \label{th:strip}
In the upper half space model of $\h^3$, let $K$ be the strip
\[
  [-1,1]\times \RR \times \{0\} = \{(x,y,z): |x|\le 1, z=0\}
\]
together with the point at infinity.

Then there is a unique area-minimizing surface $M\subset \h^3$ with boundary $\partial K$,
and $M$ has the form
\[
   \{(x,y,z):  z = u(x), |x|<1\}
\]
where $u:(-1,1)\to \RR$ is a smooth function such that
\begin{align*}
u''&<0,
\\
u(x)&\equiv u(-x), 
\\
\lim_{x\to \pm 1} u(x)&=0.
\end{align*}
Furthermore, the surface $M$ is strictly $L^\infty$ stable.
\end{theorem}

\begin{definition}\label{def:strip}
The surface $M$ in Theorem~\ref{th:strip} will be called the {\em standard minimal strip}.
A surface related to $M$ by an isometry of $\h^3$ will be called a {\em minimal strip}.
\end{definition}

\begin{proof}[Proof of Theorem~\ref{th:strip}]
Note that each of the planes $x=1$ and $x=-1$ is uniquely area minimizing.
However:

\begin{claim*}
For $a>0$, let $P_a$ be the pair of planes $x=a$ and $x=-a$.
Then $P_a$ is not area-minimizing.
\end{claim*}

To prove the claim,
note that $P_a$ and $P_{\lambda a}$ are related by the hyperbolic isometry
\[
(x,y,z)\mapsto (\lambda x, \lambda y, \lambda z). 
\]
  Thus it suffices
to prove the claim for one value of $a$.
Let $C$ be a solid Euclidean cylinder in $\{(x,y,z):z>0\}$ that is perpendicular
to the planes $x=\pm a$.
Note that the hyperbolic area of the two disks $P_a\cap C$ is independent of $a$,
but that the hyperbolic area of the annular portion of $C$ between the two
planes $y=\pm a$ tends to $0$ as $a\to 0$.
Thus for small $a$, the pair $P_a\cap C$ is not area-minimizing,
which implies that $P_a$ is not area-minimizing, proving the claim.

Now suppose $M$ is an area-minimizing surface with boundary $\partial P_a$.
If $M$ were not connected, it would be equal to $P_a$ since the planes $x=a$
and $x=-a$ are each uniquely area minimizing, contradicting the claim.  Thus $M$
must be connected.

Let $M_{\rm in}$ and $M_{\rm out}$ be the innermost and outermost least area surfaces with
boundary $\partial K$, as in Theorem~\ref{th:E}.   As we have just seen, $M_{\rm in}$
and $M_{\rm out}$ are connected.

Then (see Remark~\ref{re:unique}), 
$M_{\rm in}$ and $M_{\rm out}$ are both invariant under translations $(x,y,z)\mapsto (x,y+c,z)$.
It follows that 
\[ 
  M_{\rm in}= \Pi^{-1}C_{\rm in}
\]
and
\[
  M_{\rm out} = \Pi^{-1} C_{\rm out},
\]
where $\Pi: (x,y,z)\mapsto (x,z)$ and 
where $C_{\rm in}$ and $C_{\rm out}$ are smooth curves in $\{(x,z): z>0\}$
joining $(-1,0)$ to $(1,0)$.

Now if $M_{\rm in}\ne M_{\rm out}$, there is some $\lambda>0$ such that $\lambda C_{\rm in}$ intersects
$C_{\rm out}$.  Thus there is a largest $\lambda$ (since $C_{\rm in}$ and $C_{\rm out}$ have the same
endpoints and have compact closures.)  But then $\lambda M_{\rm in}$ and $M_{\rm out}$ violate the maximum
principle. 

Thus there is a unique least area surface $M=M_{\rm in}=M_{\rm out}$ with boundary $\partial K$.

Now where  the tangent to the curve $C=C_{\rm in}=C_{\rm out}$ is not vertical, it is locally the graph
of a function $z=u(x)$ that satisfies a 
2nd order ODE\footnote{Note that $(x,y)\mapsto u(x)$ is a solution of the Euler-Lagrange
equation for the hyperbolic area functional. That is a 2nd order PDE, but since $u$ is a function
of $x$ alone, the PDE reduces to an ODE.}, namely 
\[
u(x) \cdot u''(x)+2 (1+(u'(x)^2)=0, 
\]
from which we see that $u''<0$
and thus that $C$ has the form 
\[
  C=\{(0,y,u(y)) \; : \; |y |<1\}, \quad \lim_{y\to \pm 1} u(y)=0.
\]
By Remark~\ref{re:unique}, $M$ is invariant under $(x,y,z)\mapsto (-x,y,z)$
and hence the function $u$ is even.

So, summarizing all the information that we have, we 
are able to deduce that $x \cdot u'(x) \le 0$ for $-1<x<1$. Furthermore, we know that 
\[
\lim_{x \to -1} u'(x)=+\infty, \quad \mbox{and} \quad \lim_{x \to +1} u'(x)=-\infty.
\]

{ Let $w^*$ be the Jacobi field on $M$ associated
to dilations $(x,y,z)\mapsto \lambda(x,y,z)$.  
 Note that $w^*(x,y,z)$ is independent of $y$:
\begin{equation}\label{eq:independent-of-y}
  w^*(x,y,z)=w^*(x,0,z).
\end{equation}
Note also that  $w^*$ is strictly positive everywhere, so compact domains in $M$ are strictly stable. 
A straightforward computation 
gives
\[
 w^*= \frac{-x u'+u}{u \sqrt{1+(u')^2}},
\]
so
\begin{equation}\label{strip-boundary-blowup}
\text{$w^*\to\infty$ uniformly as $x\to \pm 1$.}
\end{equation}
  Now suppose that $M$ is not $L^\infty$ strictly stable, i.e., that $M$ 
  has a bounded, nonzero Jacobi field $v$.    We may assume that $v$ is strictly
  positive at some points.
  Let $\Lambda$ be the supremum of $v/w^*$, and let 
  $p_n:=(x_n,y_n,z_n)\in M$ be a sequence of points such that 
\[
  v(p_n)/w^*(p_n) \to \Lambda.
\]
By~\eqref{strip-boundary-blowup}, the $|x_n|$ is bounded away from $1$.
Thus by passing to a subsequence, we can assume that the points
$(x_n,0,z_n)$ converge to a point $p\in M$ and  that the Jacobi fields
$(x,y-y_n,z)\mapsto v(x,y,z)$ converge smoothly to a limit Jacobi field $\hat{v}$.
Note that $\hat{v}/w^*$ attains its maximum value $\Lambda$ at $p$.  Thus 
the Jacobi field $\hat{v}-\Lambda \cdot w^*$ attains
its maximum value, namely $0$, at $p$.  By the maximum principle, 
   $\hat{v}- \Lambda \cdot w^*$
must be identically $0$.  But that is impossible since $\hat v$ is bounded and $ \Lambda w^*$ is unbounded.}
\end{proof}

\begin{definition}[Skillet]
Suppose $u: \RR \to [0,+\infty]$ is a continuous, compactly supported function such that $u(x)=\infty$ if and
only if $|x|\le 1$ and such that $\mathcal{A}=\{ (x,y) \in \R^2 \; : \; y \le u(x) \}$ has a uniformly smooth boundary, with $u''(x)\ge 0$ along the boundary of $\mathcal{A}$ (see Fig.~\ref{fig:skillet}.)
Then the set $\mathcal{A}$ is called a {\em skillet}.
\end{definition}

\begin{figure}[htbp]
    \begin{center}
        \includegraphics[width=.75\textwidth]{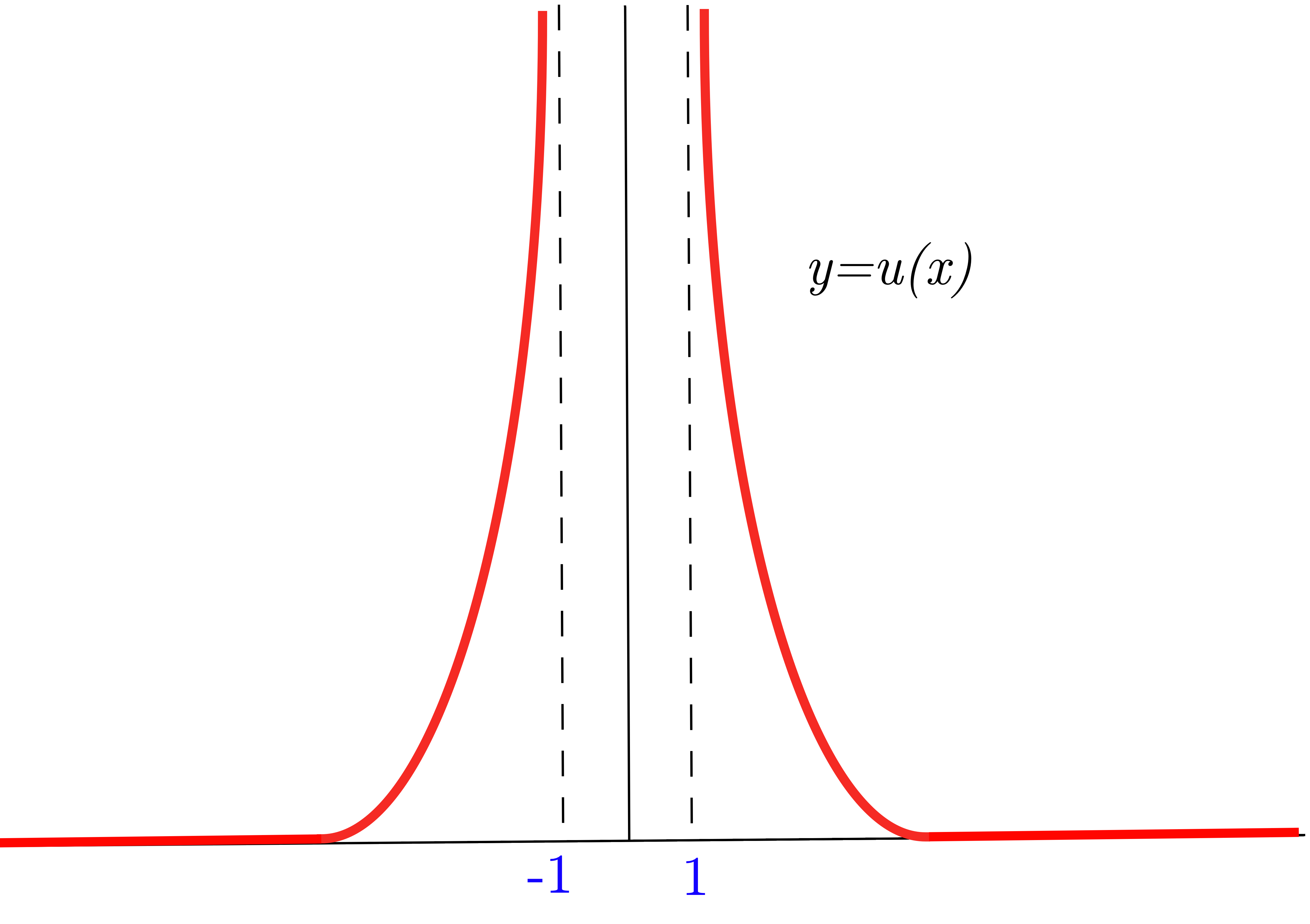}
        \includegraphics[width=.75\textwidth]{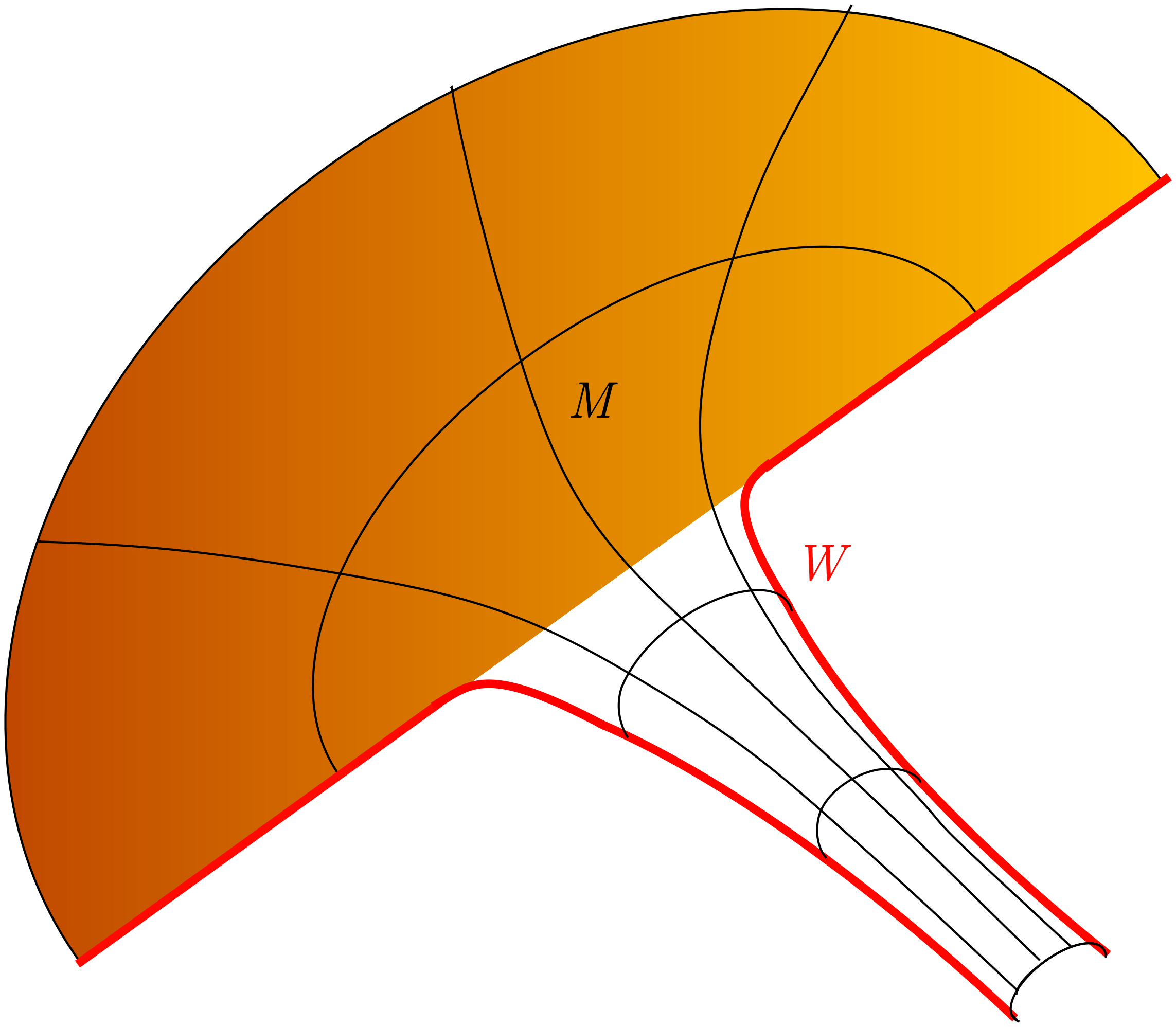}
    \end{center}
   \caption{The boundary of a skillet and the minimal skillet $M$} \label{fig:skillet}
\end{figure}


\begin{theorem} \label{th:skillet}
Let $\mathcal{A}$ be a skillet in $\h^3$.
Then there exists a  properly embedded, uniquely area-minimizing
surface $M$ satisfying $\partial M=\partial \mathcal{A}$. 
The surface is a radial graph in the following sense: if $p=(0,y_p,0)$ with $y_p<0$,
if $H$ is the vertical halfplane $\{(x,0,z):z>0\}$, 
and if 
\begin{align*}
&\Pi: M \to H \\
&\Pi(q)= \overleftrightarrow{pq}\cap H,
\end{align*}
then $\Pi$ is a diffeomorphism.
Furthermore, $M$ has a normal vectorfield $\nu$ such that $\nu\cdot (0,1,0)$ is everywhere
strictly positive.
\end{theorem}

{
\begin{definition}
The minimal surface $M$ in Theorem~\ref{th:skillet} is called a {\em minimal
skillet}.
\end{definition}}

\begin{proof}
By Theorem~\ref{th:principal}, there exists a properly embedded, area-minimizing surface $M$ with
$\partial M = \partial \mathcal{A}$.   Furthermore, $\overline{M}$ is a smooth, embedded manifold
with boundary except at the point at infinity (where $\partial\mathcal{A}$ is not smooth).

\begin{claim}\label{skillet-asymptotics-claim}
The surface $M$ is asymptotic to the standard minimal strip (see Definition~\ref{def:strip}) as $y\to \infty$,
and is asymptotic to the geodesic plane $H$ as $|x|\to\infty$.
In other words, 
if $M-(x,y,0)$ is the result of translating $M$ by $-(x,y,0)$, then $M-(0,y,0)$ converges smoothly to the 
standard minimal strip as $y\to\infty$, 
and $M-(x,0,0)$ converges smoothly to $H$ as $|x|\to \infty$.
\end{claim}

This claim follows immediately from the fact that 
the standard minimal strip and the totally geodesic plane $H$ are uniquely area minimizing
(by Theorem~\ref{th:strip} and by the convex hull property~\ref{convex-hull-proposition}.)

\begin{claim}\label{skillet-containment-claim}
The surface $M$ lies in the region $\{y\ge 0\}$.  Also, there is an $a>0$ such that
\[
  M\cap \{y>a\}
\]
lies in a cylinder $x^2+z^2\le r^2$.
Furthermore, as $\lambda\to 0$, the surface
\[
    \lambda(M\cap \{x^2+z^2\ge r^2\})
\]
converges smoothly on compact subsets of $\overline{\h^3}\setminus \{0\}$
to $\overline{H}$.
(Here $\lambda(S)$ denotes the result of dilating $S$ by $\lambda$ about the origin.)
\end{claim}

\begin{proof}
The first statement follows from the convex hull property (Proposition~\ref{convex-hull-proposition}).
To prove the second, note that (by Claim~\ref{skillet-asymptotics-claim}) 
we can choose $a>0$ so that one component 
of $M\cap \{y>a\}$ lies in a bounded Euclidean distance from the standard minimal strip
and hence lies in a cylinder $x^2+z^2\le r^2$.  If $M\cap\{y>a\}$ had another component $\Sigma$,
then $\Sigma$ would be an open minimal surface whose limit set $L(\Sigma)$
 lies in the totally geodesic plane
$\{ (x,a,z): z\ge 0\} \cup \{\infty\}$, contradicting the convex hull property.

We have shown that the boundary of $M\cap \{x^2+z^2\ge r^2\}$
coincides (except in a ball around $(0,0,0)$) with $\partial H$.
Thus the boundary of $\lambda(M\cap\{x^2+z^2\ge r^2\})$ converges to $\partial H$, and the convergence
smooth away from the origin.
The convergence statement of the claim now follows 
from the Compactness Theorem~\ref{compactness-theorem} 
and from the fact that $H$ is uniquely area minimizing.
\end{proof}

\begin{claim}\label{disjoint-dilations-claim}
Fix a point $p$ of the form $(0,y_p,0)$ with $y_p<0$, and for $\lambda>0$, let $M^\lambda$ be the result of dilating $M$ by 
$\lambda$ about the point $p$.   
Suppose $N$ is another area-minimizing surface such that $\partial N = \partial M$.  
Then  $M^\lambda$  is disjoint from $N$ for $\lambda\ne 1$.
\end{claim}

\begin{proof}
It suffices to prove the claim for $\lambda<1$, since the result for $\lambda>1$ follows by switching the roles
of $M$ and $N$.  

Since $N$ is properly embedded in $\h^3$, by elementary topology we can write $N$ as the boundary of an open region $U$
of $\h^3$.  We may assume that $(0,0,0)$ is not in $\overline{U}$.  (Otherwise replace $U$ by
the interior of $\h^3\setminus U$.)

Note that if $\lambda<1$ and if $M^\lambda$ intersects $N$, then there are points of $M^\lambda\cap N$
where the intersection is transverse, from which it follows that if we perturb $\lambda$ slightly,
$M^\lambda$ still intersects $N$.  Thus if $\Lambda$ is the set of  $\lambda\in (0,1)$ for
which $M^\lambda$ intersects $N$, then $\Lambda$ is open.

Note that there is an $R>0$ with the following property:
if $\lambda\in \Lambda$, then $M^\lambda\cap N$ contains points in the cylinder $C=\{(x,y,z): x^2+z^2\le R^2\}$.
To see this, first choose $R$ larger than the $r$ of claim~\ref{skillet-asymptotics-claim}, 
from which it follows
that  $\overline{N}\setminus C$ is a smooth manifold-with-boundary near $\infty$.   
Now choose $R$ even larger so that $N':=\overline{N}\setminus C$ has
the strong local uniqueness property described in Theorem~\ref{strong-local-uniqueness-theorem}.
  If $M^\lambda\cap N$ did not have any points in $C$,
then $L(M^\lambda\cap U)$ would be contained in $N'$, and therefore 
(by Theorem~\ref{strong-local-uniqueness-theorem}), $M\cap U$ would be contained in $N'$,
a contradiction.

We claim that $\Lambda$ is also relatively closed in $(0,1)$. 
For suppose $\lambda(i)\in \Lambda$ converges to $\lambda\in (0,1)$.
By the preceding paragraph, there exist points $(x(i),y(i),z(i))$ in $M^{\lambda(i)}\cap N$
with $x(i)^2+z(i)^2\le R^2$.
It follows from claim~\ref{skillet-asymptotics-claim} (applied to $M$ and to $N$) that
 $z(i)$ is bounded away from $0$.  (Note that $\partial M^\lambda$ and $\partial N$
 are a positive Euclidean distance apart.)
Also, $y(i)$ is bounded since (by claim~\ref{skillet-asymptotics-claim})
 $M$ and $N$ are both asymptotic to 
the standard minimal strip  as $y\to\infty$, and therefore that $M^\lambda$ and $N$ are 
a positive distance apart as $y\to\infty$.
Hence, after passing to a subesquence, $(x(i),y(i),z(i))$ converges to a point $p\in M^\lambda\cap N$, proving
that $M^\lambda\cap N$ is nonempty and thus that $\Lambda$ is relatively closed in $(0,1)$.

Since $\Lambda$ is an open and closed subset of $(0,1)$, 
either it is either empty or else it is all of $(0,1)$.
To see that it is empty, note that $M^\lambda$ is disjoint for $N$ for very small $\lambda$ since,
by claim~\ref{skillet-containment-claim},
\[
    \max_{q\in M^\lambda} \dist_{\RR^3} (q, T) \to 0 \quad \text{as $\lambda\to 0$}
\]
and
\[
   \min_{q\in N} \dist_{\RR^3}(q, T) > 0
\]
where $T$ is the union of $\{y=y_p, \, z\ge 0\}$ and $\{(0,y,0): y\ge y_p\}$.
This completes the proof of claim~\ref{disjoint-dilations-claim}.
\end{proof}

Now we can complete the proof of Theorem~\ref{th:skillet}.
It follows immediately from Claim~\ref{disjoint-dilations-claim} that $M$ is unique.

Let $H=\{(x,0,z): z>0\}$ and let  
\begin{align*}
&\Pi: M\to H \\
&\Pi(q) =  \overleftrightarrow{pq} \cap H.
\end{align*}
Applying claim~\ref{disjoint-dilations-claim} with $N=M$, we see that each straight (Euclidean) line
through $p=(0,y_p,0)$ intersects $M$ at most once.  Thus the map $\Pi$ is a
diffeomorphism from $M$ to an open subset of $\Omega$.
It follows from Claim~\ref{skillet-asymptotics-claim} that $\Pi$ is proper.  Hence $\Pi$ is a surjective diffeomorphism.

It follows that the Jacobi field on $M$ corresponding to dilations about $p=(0,y_p,0)$
is everywhere positive.  Letting $y_p\to -\infty$, we see that the Jacobi field on $M$
corresponding to horizontal translations $(x,y,z)\mapsto (x,y+t,z)$ is everywhere nonnegative.
By the strong maximum principle, if that Jacobi field vanished anywhere, it would vanish everywhere,
which implies that $M$ would be invariant under those translations.  But that is impossible
since $y\ge 0$ for $(x,y,z)\in M$.    Thus the Jacobi field is everywhere positive, which
implies that $\nu\cdot (0,1,0)$ is everywhere positive.  
\end{proof}

\begin{theorem}
A {minimal} skillet $M$ in $\h^3$ is strictly $L^\infty$ stable.
\end{theorem}

\begin{proof}
We will assume that the minimal skillet has been translated by $(x,y,z)\mapsto (x,y+1,z)$,
so that it 
lies in the region $\{(x,y,z): \text{$z>0$ and $y>1$}\}$ and is 
asymptotic as $x^2+z^2\to \infty$ to the halfplane
               $y=1, z\ge 0$.
As $y\to\infty$ with $x^2+z^2$ bounded, the minimal skillet $M$ is smoothly asymptotic to the standard minimal
strip.  (See Claim~\ref{skillet-asymptotics-claim}.)

Let $w$ be the Jacobi field on $M$ corresponding to dilations about $0$.
In other words, for $p\in M$, $w(p)$ is the (hyperbolic) length of $p^\perp$.
Then because $M$ is a radial graph about the origin (by~Theorem~\ref{th:skillet}),
$w>0$ everywhere, so compact subsets of $M$ are strictly stable.
Thus it suffices to show that $M$ has no nonzero, bounded Jacobi fields.

Suppose to the contrary that $v$ is a nonzero, bounded Jacobi field.

\begin{claim}\label{claim-1}   $z w(x,y,z)$ is bounded away from $0$.
\end{claim}

To prove the claim, note that $w(x,y,z)$ is the hyperbolic length of the vector $(x,y,z)^\perp$
at the point $(x,y,z)$, 
so $z w(x,y,z)$ is the Euclidean length $|(x,y,z)^\perp|$ of $(x,y,z)^\perp$.

As $x^2+z^2\to\infty$ in $M$, $\Tan_{(x,y,z)}M$ converges to the plane $y=0$,
so $(x,y,z)^\perp \sim (0,y,0)$.  Also, $y\ge 1$ on $M$, so
\[
  \liminf_{x^2+z^2\to\infty} z w(x,y,z) \ge 1.
\]

On sets where $x^2 + z^2$ and $y$ are both bounded, the euclidean length
of $(x,y,z)^\perp$ is bounded away from $0$ because $M$ is a radial graph.

Thus it remains to show that the Euclidean length $|(x,y,z)^\perp|$ is bounded 
as $y\to\infty$ with $x$ and $z$ bounded.  
But that holds because
that $M$ is asymptotic as $y\to\infty$ to the standard minimal strip 
(see Claim~\ref{skillet-asymptotics-claim}) and because
the corresponding Jacobi field $w^*$ on the standard minimal strip is bounded
away from $0$ (by~\eqref{eq:independent-of-y} and~\eqref{strip-boundary-blowup}).
This completes the proof of the Claim~\ref{claim-1}.

\begin{claim}\label{claim-2}
 If $p_n=(x_n,y_n,z_n)$ is a divergent sequence in $M$, then $v(p_n)\to 0$.
\end{claim}

\begin{proof}[Proof of Claim~\ref{claim-2}]
 By passing to a subsequence, we can assume that one of the following holds:

\begin{enumerate}
\item $(x_n)^2+ (z_n)^2 \to \infty$.
\item $(x_n)^2 + (z_n)^2$ is bounded and $z_n\to 0$.
\item $(x_n)^2 + (z_n)^2$ is bounded and $z_n$ is bounded away from $0$.
\end{enumerate}

Translate $M$ by $(-x_n,-y_n,0)$ and then dilate by $1/z_n$
to get a surface $M_n$.  Let $v_n$ be the Jacobi field on $M_n$ corresponding to 
   $v$ on $M$.
By passing to a subsequence, we can assume that $M_n$ converges smoothly
to a limit surface $\hat{M}$, and that $v_n$ converges to a bounded Jacobi field $\hat{v}$
on $\hat{M}$.  
In case (1), $\hat{M}$ is the vertical halfplane $\{y=0\}$ (by Claim~\ref{skillet-containment-claim}).
In case (2), $\hat{M}$ is also a vertical halfplane, 
since $\partial \hat{M}$ 
is a line in $\partial \h^3$.  In case (3), $\hat{M}$ is  
a minimal strip (see Definition~\ref{def:strip}) by Theorem~\ref{th:strip}.   
In all three cases, $\hat{M}$ is strictly stable.  Thus $\hat{v}=0$.
Since $\hat{v}(0,0,1)= \lim v_n(0,0,1) = \lim v(p_n)$, 
this completes the proof of Claim~\ref{claim-2}. 
\end{proof}

\begin{claim}\label{claim-3}
There exists a Jacobi field $f$ on $M$ and an $R>0$
such that 
\[
    \inf_{M\cap \{x^2 + z^2 > R^2\}} f > 0.
\]
\end{claim}

\begin{proof}[Proof of Claim~\ref{claim-3}]
Let $S=S_R$ be the surface obtained from
$\overline{M}\cap \{ x^2+ z^2 > R^2\}$ by inversion in the sphere $x^2+y^2+z^2=1$
 (where $\overline{M}$ denotes the
 closure of $M$ in~$\overline{\h^3}$).
Note that if $R$ is sufficiently large, then $S$ is a smooth manifold-with-boundary
 on which $y$ is a smooth, function of $x$ and $z$.
  Indeed, by choosing
$R>0$ sufficiently large, we can guarantee that the Euclidean unit normal
to $S$ is everywhere arbitrarily close to $(0,1,0)$.  Consequently, 
the Jacobi field corresponding to translations in the $y$-direction is bounded 
away from $0$ on $S$.
  Now let $f$ be the corresponding jacobi field on $M$.  This completes the proof 
  of Claim~\ref{claim-3}.
\end{proof}

Now let $\lambda=\sup(v/w)$.  
Since we are assuming that $v>0$ at some points, $\lambda>0$.
By Lemma~\ref{ratio-lemma}, the supremum is not attained at any point of $M$.
(Note that $v$ cannot be a multiple of $w$ since $v$ is bounded and $w$ is unbounded.)
Thus if $p_n=(x_n,y_n,z_n)$ is a sequence of points in $M$ with
$v(p_n)/w(p_n)\to \lambda$, then $p_n$ diverges in $M$.  
By 
Claim~\ref{claim-2},
 $v(p_n)\to 0$.  Since $\lambda>0$, this implies that $w(p_n)\to 0$,
and therefore by Claim~\ref{claim-1} that $z_n\to \infty$.

It follows that by choosing $\mu<\lambda$ sufficiently close to $\lambda$, we
can guarantee that 
\[
    (v - \mu w)^+
\]
is supported in $M\cap \{z>R\}$, where $R$ is as in Claim~\ref{claim-3}.   It follows (using
Claims~\ref{claim-2} and~\ref{claim-3}) that
\[
    (v-\mu w)^+ / f
\]
attains a positive maximum value $k$ at some point $p$.
Consequently, 
\[
   (v - \mu w)/f
\]
has a positive local maximum $k$ at $p$, 
so
\[
  v - \mu w - k f \equiv 0
\]
by Lemma~\ref{ratio-lemma}.
But that is impossible since $(v-\mu w)$ is negative at some points of 
  $M\cap \{z>R\}$ whereas $f>0$ everywhere on that set.
The contradiction proves that there is no such $v$, and therefore that
$M$ is strictly $L^\infty$ stable.
\end{proof}

\section{Bridge principle at infinity}\label{sec:bridge}

A key tool in the construction of our minimal embeddings with arbitrary topology
is a bridge principle at infinity for properly embedded area-minimizing surfaces
in $\h^3$. 

Let $M\subset \h^3$ be an smooth, properly embedded, open surface whose closure
$\overline{M}$ is a smooth manifold-with-boundary in $\overline{\h^3}$.
Let $\Gamma \subset \partial\h^3$ be a smooth embedded arc such that 
 $\overline{M} \cap \Gamma= \partial \Gamma$ and such that $\Gamma$ meets $\partial  M$ orthogonally at each of its ends points. 
 A {\bf bridge on $M$ along $\Gamma$} is the image $P$ of a homeomorphism
$$\phi: [0,1] \times [-1,1] \longrightarrow \partial\h^3$$
such that $\phi( \cdot,0)$ parametrizes $\Gamma$ and $\phi(t,s) \in \overline{M}$ if, and only if, $t=0$ or $t=1$.

By the (Euclidean) {\bf width} of $P$ we mean $$w(P)=\sup_{x \in P} \dist_{\rth}(x, \partial P).$$

For the following proposition, we shall consider the {\em half-space model} of $\h^3$. In this model, the homotheties centered at points $p \in \{z=0\}$ induce isometries of the hyperbolic space.

From now on, it will be convenient to generalize the notions of skillet and minimal skillet as follows:
the image of a skillet (or minimal skillet) under any isometry of $\h^3$ leaving $\infty$
fixed will also be called a skillet (or minimal skillet).

\begin{proposition} \label{pro:bridges}
Let $M$ and $\Gamma$ be as be as above (in the preceding four paragraphs). 
Then there exists a sequence of bridges $\{ P_n\}_{n \in \n}$ on $M$ along $\Gamma$ satisfying:
\begin{enumerate}[(a)]
\item The widths $w_i:=w(P_i)$ tend to $0$ as $i \to \infty;$
\item The symmetric difference $(\partial P_i ) \vartriangle \partial M$ is smooth, and if $x_i \in P_i$, then every sequence of $i$'s tending to $\infty$ has a subsequence $i(j)$ such that
\[
      \left(w_{i(j)}^{-1}\right)_ \# \left( (\partial P_{i(j)}) \vartriangle \partial M  - x_{i(j)}\right)
\]
converges smoothly on compact sets of $\overline{\h^3}\setminus\{\infty\}$ to either:
\begin{enumerate}[(1)]
\item two parallel straight lines, or
\item the boundary of a {\bf skillet}.
\end{enumerate}
\end{enumerate}
\end{proposition}

Recall that $A \vartriangle B := (A \setminus B) \cup (B \setminus A)$.

The proof of Proposition~\ref{pro:bridges} is straightforward so we omit it.
A sequence of bridges $P_i$ that satisfies the conclusions of Proposition~\ref{pro:bridges} is said
to {\em shrink nicely} to $\Gamma$. 

\begin{theorem}[Bridge Theorem] \label{th:bridge}
Let $S\subset\h^3$ be an open, properly embedded, uniquely area-minimizing surface whose
closure $\overline{S}\subset \overline{\h^3}$ is a smooth, embedded manifold-with-boundary.
 Let $\Gamma$ be a smooth arc in $\partial\h^3$ meeting
 $\partial S$ orthogonally and satisfying $\Gamma \cap \partial S=\partial \Gamma$. 
  Consider a sequence of bridges $P_n$ 
 in $\partial\h^3$ that shrink nicely to $\Gamma$. If $S$ is {\bf strictly $L^\infty$ stable}, then for all large enough $n$, there exists a strictly $L^\infty$ stable,  uniquely area-minimizing surface $S_n$ {that} is properly embedded in $\h^3$ and {that satisfies}:
 \begin{enumerate}
 \item\label{boundary-specification-item}
  $\partial S_n= \partial S \vartriangle \partial P_n$ (in particular, 
 $\overline{S_n}$ is  smooth, embedded manifold-with-boundary in $\overline{\h^3}$);
  \item\label{away-from-Gamma-item}
   The sequence $\overline{S_n}$ converges smoothly to $\overline{S}$ on 
                   compact subsets of $\overline{\h^3}\setminus\Gamma$;
  \item\label{homeomorphic-item} The surface $\overline{S_n}$ is homeomorphic to $\overline{S} \cup P_n$
  \end{enumerate} 
\end{theorem}

Such a sequence of bridges exists by Proposition~\ref{pro:bridges}. 

\begin{proof}
By Theorem~\ref{th:principal}, there is an area-minimizing surface $S_n$ satisfying~\ref{boundary-specification-item}.
By the Compactness Theorem~\ref{compactness-theorem}, every subsequence of $S_n$
has a further subsequence such that $\overline{S_n}$ converges smoothly on
compact subsets of $\overline{\h^3}\setminus \Gamma$ to $\overline{Q}$, where $Q$
is an area-minimizing surface with boundary $\partial S$.  Since $S$ is uniquely area-minimizing,
in fact $Q=S$, which proves~\ref{away-from-Gamma-item}.

The key to proving the rest of the Bridge Theorem is the following:

\begin{claim}\label{possible-limits-claim}
Let $(x_n,y_n,z_n)$ be a sequence of points in with $z_n>0$ and $z_n\to 0$.
 Translate $S_n$ by $-(x_n,y_n,0)$ and then dilate by $1/z_n$
to get a surface
\[
  S_n' := (S_n - (x_n,y_n,0))/ z_n.
\]
Then a subsequence of the  $S_n'$ converges smoothly on compact subset of $\h^3$
 to one of the following
surfaces $S'$:
a vertical half-plane,  
a minimal skillet, a minimal strip (see Definition~\ref{def:strip}), or the empty set.
In particular, $S'$ is uniquely area-minimizing and strictly $L^\infty$ stable.

Furthermore, if $T_n$ is another area-minimizing surface with $\partial T_n=\partial S_n$,
and if $T_n' = (T_n - (x_n,y_n,0))/z_n$, then the corresponding subsequence
of the $T_n'$ converges to the same limit surface $S'$.
\end{claim}

\begin{proof}[Proof of Claim~\ref{possible-limits-claim}]
By the definition of nicely shrinking, after passing to a subsequence, the curves $\partial S_n'$
converge to a limit $C'$, where
 $C'$ is one of the following configurations together with the point at infinity:
\begin{enumerate}
\item\label{line-item} a straight line with multiplicity $1$.
\item\label{T-item} a $T$-shaped configuration consisting of a straight line with multiplicity $1$ together with a   perpendicular half-line with multiplicity $2$.
\item\label{skillet-item} the boundary of a skillet with multiplicity $1$. 
\item\label{parallel-item} two parallel lines, each with multiplicity $1$.
\item\label{double-line-item} a straight line with multiplicity $2$.
\item\label{emptyset-item} the empty set.
\end{enumerate}
The convergence is smooth except at $\infty$ and, in case~\ref{T-item}, at the vertex of the $T$.

By the Compactness Theorem~\ref{compactness-theorem}, 
the $S_n'$ converge smoothly
(after passing to a subsequence)
to a limit surface $S'$ whose boundary is a closed subset of $C'$ that contains the multiplicity $1$
portion of $C'$ but none of the multiplicity $2$ portion.  
Thus (in all these cases) $\partial S'$ is the closure
of the multiplicity $1$ portion of $C'$.
In particular,  $\partial S'$ is a straight line in cases~\ref{line-item} and~\ref{T-item}, 
a skillet boundary in case~\ref{skillet-item}
 and a pair of parallel lines in case~\ref{parallel-item}.  
It follows that $S'$ is a vertical halfplane, a minimal skillet,
or a minimal strip since each of those surfaces is uniquely area-minimizing.  

In cases~\ref{double-line-item} and~\ref{emptyset-item}, $\partial S'$ is the empty set, which implies (by the convex hull property)
that $S'$ is empty.

By passing to a further subsequence, we can assume that the $T_n'$ converge smoothly
on compact subsets of $\h^3$ to a surface $T'$ whose boundary is (as proved above)
the closure of the multiplicity $1$ portion of $C'$.  In other words, $T'$ and $S'$ have the same
boundary.  Since in each of the cases above, $S'$ is uniquely area-minimizing,
it follows that $T'=S'$.
\end{proof}

\newcommand{\uu}{\mathbf{u}}
Next we shall prove that $\overline{S_n}$ and $\overline{S} \cup P_n$ are homeomorphic. 
{The surface $\overline S_n$ separates $\overline  \h^3$ into two connected components, 
one of which contains the curve $\Gamma$ which we denote by $\mathcal{Q}_n$.}

For $a>0$, we define $\mathcal{R}_a:= \{ (x,y,z) \in \overline{\h^3} \; : \; 0 \leq z \leq a\}$. 

\begin{claim} \label{cl:a}
{There exists $a>0$ such that
$\overline{S_n} \cap \mathcal{R}_a$ does not contain any point at which 
the vector $\uu:=(0,0,1)$ is 
a normal vector to $\overline{S_n}$ that points into $\mathcal{Q}_n$.}
\end{claim}

(Thus $S_n\cap \mathcal{R}_a$ might have critical points of the height function $z$, but at such
critical points, the normal vector $(0,0,1)$ must point out of $\mathcal{Q}_n$, not into it.)

\begin{proof}
 We proceed by contradiction. Suppose this were not the case. Thus,  after passing to a subsequence, we can assume that there exists a critical point 
 $p_n=(x_n,y_n,z_n) \in S_n$ with $\uu$ pointing into $\mathcal{Q}_n$ at $p_n$ and with $z_n \to 0$.
 Up to a subsequence, we can suppose that $\{p_n\}$ converges to some point $p_0=(x_0,y_0,0)
 \in \partial\h^3$. 
 
 Then, we apply the isometry $(x,y,z) \mapsto 1/z_n \left( (x,y,z)-(x_0,y_0,0) \right) $  to $S_n$,  $p_n$,
 $Q_n$, and $\Gamma$ to obtain a new surface $S_n'$, a point $p_n' =(0,0,1)\in S_n'$,
 a region $Q_n'$, and a curve $\Gamma_n'$.
By Claim~\ref{possible-limits-claim}, we can assume (by passing to a subsequence)
that the surfaces $S_n'$ converge smoothly to one of the following surfaces $S'$: a
vertical halfplane, a minimal skillet, or a the standard area-minimizing strip bounded by two parallel lines.
In our case, $S'$ cannot be a vertical halfplane or  a minimal skillet, because those surfaces have no
points at which $\uu=(0,0,1)$ is a normal vector (see Theorem~\ref{th:skillet}),
 whereas $\uu$ is normal to $S'$ at the point $p'=(0,0,1)$.

Thus  $S'$ is a minimal strip.  Note that the curves $\Gamma_n'$ must converge
to the straight line $\Gamma'$ that is halfway between the two lines in $\partial S'$.  It follows that the regions
$\Omega_n'$ converge to the region $\Omega'$ that lies on the other side of $S'$ from 
  $\Gamma'$.  It now follows from the description of $S'$ in Theorem~\ref{th:strip} that
the vector $\uu=(0,0,1)$ points into $\Omega'$ at $p'$.   However, the smooth convergence 
and the choice of $p_n$
imply that $\uu=(0,0,1)$ points out of $\Omega'$ at $p'$.
The contradiction proves the claim.
\end{proof}

\begin{claim} \label{cl:diffeo}
The surfaces $\overline S_n$ and $ \overline S \cup P_n$ are homeomorphic.
\end{claim}
Suppose $\overline S_n$ is not homemorphic to $\overline S \cup P_n$.
As they have the same boundary, then it means that $\overline S_n$ and $\overline S \cup P_n$
have different genus. Consider the positive constant $a$ given by Claim~\ref{cl:a}.  
The smooth convergence on compact sets implies 
$S_n \cap (\h^3 \setminus \mathcal{R}_a)$ is homemorphic to 
$S \cap (\h^3 \setminus \mathcal{R}_a)$, so our assumption 
gives that $S_n \cap \mathcal{R}_a$ has non
trivial genus. 

Up to a slight modification of the point of infinity in the upper half-space model of $\h^3$, we can assume that the function $z$ is a {\bf Morse function} for the surface $S_n$. This implies the existence of a critical point of the height function $z$ in  $\overline S_n \cap \mathcal{R}_a$ such that the vector $u=(0,0,1)$ points
in the direction of the region $\mathcal{Q}_n$, which is contrary to Claim~\ref{cl:a}. This contradiction completes the proof of this claim.

\begin{claim} \label{cl:uniqueness}
If $n$ is large enough, then the surfaces $S_n$ are uniquely area-minizing: 
if $T_n$ is any area-minimizing surface in $\h^3$ with 
$\partial  T_n= \partial S_n$, then $T_n=S_n$ (for all sufficiently large $n$). 
\end{claim}

Suppose the uniqueness is false. Then, up to a subsequence, we may assume that $S_n$ and
 $T_n$ are different for all $n$. 
 Note that all properties we have proved for $S_n$ also hold for $T_n$.
 In particular, $T_n$ also converges smoothly to $S$ on compact subsets of 
    $\overline{\h^3}\setminus \Gamma$.
 
 As $S_n$ and $T_n$ are asymptotic at $\partial\h^3$, 
 then we can find a point $p_n=(x_n,y_n,z_n) \in S_n$ that 
 maximizes\footnote{The maximum exists because the hyperbolic distance from a point $q$ in $S_n$
 to $T_n$ tends to $0$ as $q$ approaches the boundary of hyperbolic space.  This follows from
 the fact that $\overline{S_n}$ and $\overline{T_n}$ meet $\partial \h^3$ orthogonally along the same curve.}
  the (hyperbolic) distance to $T_n$.

By passing to a subsequence, we can assume that $p_n$ converges to a point $p=(x,y,z)$.
If $p\in S$, then the smooth convergence of $T_n$ and $S_n$ to $S$ give rise to a nonzero Jacobi field
on $S$ that attains its maximum absolute value at $p$.  But that is impossible by the strict $L^\infty$
stability of $S$.

Thus $p\in \partial S$, so $z_n\to 0$.
Translate $S_n$, $T_n$, and $p_n$ by $-(x_n,y_n,0)$ and then dilate by $1/z_n$ to get 
$S_n'$, $T_n'$ and $p':=(0,0,1)\in S_n'$ with
\begin{equation}\label{eq:choice}
  \dist(p',T_n') = \max_{q\in S_n'}\dist(q,T_n').
\end{equation}
By Claim~\ref{possible-limits-claim}, we can assume (by passing to a subsequence)
that $S_n'$ and $T_n'$ converge smoothly on compact subsets of $\h^3$ to the
same strictly $L^\infty$ stable limit surface $S'$.
By~\eqref{eq:choice}, the smooth convergence of $S_n'$ and $T_n'$  to $S'$
gives rise to a nonzero jacobi field on $S'$ that attains its maximum absolute value
at the point $p'$.  But that contradicts the strict $L^\infty$ stability of $S'$,
thus proving Claim~\ref{cl:uniqueness}.

To complete the proof of the Bridge Theorem~\ref{th:bridge}, it remains only to prove
that the surface $S_n$ is strictly $L^\infty$ stable for all sufficiently large $n$.
The proof is almost the same as the proof of Claim~\ref{cl:uniqueness}.
Suppose the strict $L^\infty$ stability fails.  Then we can assume that each $S_n$
has a nonzero bounded Jacobi field $V_n$.  By Theorem~\ref{th:boundary-point}, 
      $V_n(p)$ tends to $0$ as $p\to \partial S_n$,
so $|V_n(\cdot)|$ attains its maximum at a point 
$p_n=(x_n,y_n,z_n)$ in $S_n$.  We can normalize $V_n$ so that $|V_n(p_n)|=1$.
By passing to a subsequence, we can assume that $p_n$ converges to a point $p=(x,y,z)$.

If $z>0$, then the $V_n$ converge subsequentially to a Jacobi field $V$ on $S$ that attains its
maximum absolute value of $1$ at  the point $p$.  But that violates the strict $L^\infty$ stability of $S$.

Thus $z=0$.  Now translate $S_n$ by $-(x_n,y_n,0)$ and dilate by $1/z_n$ to get a surface $S_n'$.
By Claim~\ref{possible-limits-claim}, a subsequence of the $S_n'$ converges smoothly to a 
strictly $L^\infty$ stable surface $S'$.  
  However, by construction, $S'$ has a Jacobi field that attains  a maximum
absolute value $1$ at the point $p':=(0,0,1)$, a contradiction.
\end{proof}

\section{Properly embedded area-minimizing surfaces in $\h^3$}

In this section, we are going to prove the main existence results for properly embedded area
minimizing surfaces with arbitrary (orientable) topology. The techniques we use are inspired 
in those developed by Ferrer, Meeks and the first author for the study of the Calabi-Yau problem
in $\R^3$ (see \cite{fmm}).
\begin{theorem} \label{th:first}
Let $S$ be an open, connected, oriented surface. Then, there exists a complete, proper, area-minimizing embedding $\psi: S \to \h^3$. Moreover,  the embedding $\psi$ can be constructed in such a way that the limit sets of different ends of $S$ are disjoint.
\end{theorem}
\begin{proof}
Throughout this proof we are going to use the model of the Poincaré ball. 
Let $\mathcal{S}= \{S_1 \subset S_ 2 \subset \cdots \subset S_n \subset \cdots \}$  be a simple exhaustion of
$S$. Our purpose is to construct a sequence of properly embedded minimal surfaces 
$\{\Sigma_n\}_{n \in \n}$ and two sequences of positive real numbers $\{\varepsilon_n\}_{n \in \n}$  and $\{r_n\}_{n \in \n}$ satisfying:
\begin{enumerate}
\item $\{\varepsilon_n\} \searrow 0$ and $\{r_n\} \nearrow +\infty;$
\item $\displaystyle \sum_{n=1}^\infty \varepsilon_n <1$. 
\end{enumerate}
Moreover, for each $n \in \n$,  the minimal surface $\Sigma_n$ satisfies:
\begin{enumerate}[(I$_n$)]
\item $\Sigma_n$ is  strictly $L^\infty$ stable and uniquely area minimizing; 
\label{properties}
\item $\Sigma_n$ admits a $C^\infty$ extension $\overline{\Sigma}_n$ to $\overline{\h}^3$ so that $\overline{\Sigma}_n$ is diffeomorphic to $S_n$;
\item $\Sigma_n \cap \overline{B(0, r_j)}$ is diffeomorphic to $S_j$, for $ j=1, \ldots,n,$ where $B(0,r)$ represents the hyperbolic ball centred at $0$ of radius $r$;
\item $\Sigma_n \cap B(0,r_i)$ is a normal graph over its projection $\Sigma_{i,n} \subset \Sigma_i$, for $i<n$. 
Furthermore, if we write $\Sigma_n \cap B(0,r_i)=\{\exp_p \left(f_{i,n}(p) \cdot \nu_i(p) \right) \; | \; p \in \Sigma_{i,n} \}$, where $\nu_i$ is the Gauss map of $\Sigma_{i}$, then:
\begin{itemize}
\item $|\nabla f_{i,n} | \leq \sum_{k=i+1}^n \varepsilon_k$ and
\item $| f_{i,n} | \leq \sum_{k=i+1}^n \varepsilon_k$, for $i=1, \ldots,n-1$.
\end{itemize}
\end{enumerate}
First, we fix a sequence which satisfies $\displaystyle \sum_{n=1}^\infty \varepsilon_n <1$ (for instance 
$\varepsilon_n=\frac{3}{\pi^2 n^2}$).
The above sequences are obtained by recurrence. In order to define the first elements, we consider a totally geodesic disk in $\h^3$. The choice of $r_1$ is irrelevant. 

Assume now we have defined $\Sigma_n$ and  $r_{n}$ and satisfying items from (I$_n$)
to (IV$_n$).  We are going to construct the minimal surface
$\Sigma_{n+1}$. 

As the exhaustion $\mathcal{S}$ is simple, then we
know that $S_{n+1}- \Int(S_n)$ contains a unique nonannular
component $N$ which topologically is a pair of pants or an annulus
with a handle. Label $\gamma$ as the connected component of
$\partial N$ that is contained in $\partial S_n$. We label the
connected components of $\partial \Sigma_n$, $\Gamma_1, \ldots,
\Gamma_k$, in such a way that $\gamma$ maps to $\Gamma_k$ by the
homeomorphism which maps $S_n$ into $\Sigma_n$. Then, we apply
Theorem~\ref{th:bridge}  to $\Sigma_n$ in the following way.
\vskip 5mm

\noindent {\bf Case 1.} {\em $N$ is a pair of pants.} 

The curve $\Gamma_{k}$ bounds a disk $D_{k}$ in 
$\partial  \h^{3}$ that does not intersects the other boundary curves of $\Sigma_{n}$. Consider
an arc $\Gamma\subset D_{k}$ so that $\Gamma \cap \Gamma_{k}=\partial \Gamma$. Then, we apply
Theorem~\ref{th:bridge} to the configuration $\Sigma_{n} \cup \Gamma$. In this way, we construct a family $\{T_{m}\}_{m \in \n}$ of properly embedded minimal surfaces obtained from $\Sigma_{n}$ by adding a bridge $B_{m}^1$ that ``divides'' $\Gamma_{k}$ into two different curves in $\partial \h^{3} $. Note that the surfaces $T_{m}$ have the same topology as $S_{n+1},$ for all $m \in \n$.
\vskip 5mm

\noindent {\bf Case 2.} {\em $N$ is a cylinder with a handle.} 

We construct the surface 
$T_m$, like in the previous case. But this time we add a second bridge $B_m^2$
along a curve $\sigma$ joining two opposite points in $\partial B_m^1$
(see Figure~\ref{fig:ends-3}). Notice that, in this way, the old
annular component becomes an annulus with a handle. Again the resulting surfaces, that
we still call $T_{m}$, are homeomorphic to $S_{n+1}$.
\vskip 5mm

In both cases, we obtain a sequence of properly embedded, area-minimizing surfaces $T_m$ satisfying:
\begin{enumerate}[(i)]
\item $T_m$ is strictly $L^\infty$ stable and uniquely area minimizing.
\item $T_m$ admits a smooth extension to $\overline{\h}^3$ and 
$\overline{T_m}$ is diffeomorphic to $S_{n+1}$. 
\item \label{sigma-3} The surfaces $T_m \cap \overline{B(0,r_n)}$ are
diffeomorphic to $\Sigma_n \cap \overline{B(0,r_n)}$ and converge in the $C^\infty$ topology
to $\Sigma_n \cap \overline{B(0,r_n)}$, as  $m\to\infty$.   
\end{enumerate}

Item \eqref{sigma-3} and property (IV$_n$) imply that
$T_m \cap \overline{B(0,r_i)}$ can be expressed as a
normal graph over its projection $\Sigma_{i,m} \subset
\Sigma_i$, $i=1, \ldots, n$; 
$$T_m \cap
\overline{B(0,r_i)}=\{\exp_p\left(h_{m,i}(p) \, \nu_i(p)\right) \; | \; p \in
\Sigma_{i,m} \}.$$
 Since, as $m \to \infty$, the surfaces $T_m$
converge smoothly to $\Sigma_n$ in $B(0,r_n)$ and $\Sigma_n$ satisfies
(IV$_n$), then we have:
\begin{equation} \label{eq:grad}
\mbox{\rm max} \{|h_{m,i}|, |\nabla h_{m,i}| \}< \sum_{k=i+1}^{n+1} \varepsilon_k
\end{equation}
 for $m$ large enough.

Then, we define $\Sigma_{n+1}\df T_m$, where
$m$ is chosen sufficiently large in order to satisfy
\eqref{eq:grad}. We chose  $r_{n+1}$ big enough in order to guarantee
that $\Sigma_{n+1} \cap \overline{B(0,r_{n+1})}$ is diffeomorphic to $S_{n+1}$. It is
clear that $\Sigma_{n+1}$ so defined fulfills (I$_{n+1}$), $\ldots,$  (IV$_{n+1}$).
\begin{remark} \label{sei}
Taking into account the way in which we are using the bridge principle at infinity to modify the topology of $\Sigma_n$, it is important to notice that the new boundary curves of $\Sigma_{n+1}$ are contained in the disk $D_k \subset  \partial\h^3$.
\end{remark}

Now, we have constructed our sequence of minimal surfaces $\{
\Sigma_n \}_{n \in \n}$. Taking into account properties (IV$_n$), for
$n \in \n$, and using Ascoli-Arzela's theorem, we deduce that the
sequence of surfaces $\{ \Sigma_n \}_{n \in \n}$ converges to a
properly embedded minimal surface $\Sigma$ in the $C^m$ topology, for
all $m \in \n$. Moreover, $\Sigma \cap \overline{B(0,r_i)}$ is a normal
graph over its projection $\Sigma_{i,\infty} \subset \Sigma_i$, for all $i \in \n$, and the
norm of the gradient of the graphing functions its at most $1$ (see
properties (IV$_n$)).

Finally, we check that $\Sigma$ satisfies all the statements in the
theorem.
\vskip 3mm

\noindent $\bullet$ {\em $\Sigma$ is diffeomorphic to $S$.} If we
consider the (simple) exhaustions $\{\Sigma \cap \overline{B(0,r_n)} \; | \; n
\in \n\}$ of $\Sigma$ and $\{ S_n \; | \; n \in \n \}$ of $S$, then
we know that there exists a diffeomorphism $\psi_n: S_n \to \Sigma \cap \overline{B(0,r_n)}. $
Furthermore, due to the way in which we have constructed $\Sigma$, we have  that
$\psi_n |_{S_i}=\psi_i$, for all $i<n$. Hence, we can construct a diffeomorphism $\psi: S\to \Sigma$. 

If we consider on $S$ the pull back of the metric of $\Sigma$, then $\psi$ is the minimal embedding
we are looking for.
\vskip 3mm

\noindent $\bullet$ {\em $\Sigma$ is area minimizing.} {The limit of area-minimizing surfaces
is area minimizing.}
\vskip 3mm

\noindent $\bullet$ {\em The limit sets of distinct ends are disjoint.} 
We are going to assume that $\Sigma$
has at least two ends, as otherwise this property does not make sense. 
 Two different ends of $\Sigma$, $E_1$ and $E_2$, can be represented by
 two disjoint components, $C_1$ and $C_2$, of $\Sigma \setminus B(0,r_n)$, for a sufficiently large 
 $n \in \n$. Consider $\partial_i=C_i \cap \overline{B(0,r_n)}$, $i=1,2$. Recall that $\Sigma \cap
 \overline{B(0,r_n)}$ is a graph over $\Sigma_n$. Then, we label $\partial_1^n$ and $\partial_2^n$ 
 the projection over $\Sigma_n$ of $\partial_1$ and $\partial_2$, respectively.
 
 Observe that, from our method of construction, $\partial_i$ (and $\partial_i^n$) is a connected curve,
 for $i=1,2$. The curves $\partial_1^n$ and $\partial_2^n$ bound two different annular ends of $\Sigma_n$
 that we call $A_1^n$ and $A_2^n$, respectively. 
 For $i=1,2$, let $\Gamma_i^n$ be the ideal boundary of
 $A_i^n$:
 \[
     \Gamma_i^n := \overline{A_i^n}\cap\partial \h^3.
\] 
The curve $\Gamma_i^n$ bounds a disk $D_i^n \subset \partial\h^3$,
 $i=1,2$, and we know that $D_1^n \cap D_2^n = \varnothing$. Taking Remark~\ref{sei} into  account,
 we deduce that $L(E_1)  \subset  D_1^n$ and $L(E_2) \subset D_2^n$. This concludes the proof.
 \end{proof}

We would like to finish this section by pointing out that a suitable modification of the 
 methods allow us to construct
properly embedded area-minimizing surfaces so that the limit set is the whole ideal
boundary $\partial\h^3$.

\begin{lemma}\label{far-apart-lemma}
If $R/r$ is sufficiently large, then there is no open, connected, area-minimizing
surface $M$  in $\h^3$ such that, in the upper halfspace model of $\h^3$,
\begin{enumerate}[\rm (i)]
\item $\partial M$ is disjoint from $\{p\in\overline{\h^3}: r< |p| < R\}$, and
\item $M\cap\{|p|\le r\}$ and $M\cap \{p: |p|\ge R\}$ are both nonempty.
\end{enumerate}
\end{lemma}

Here $|p|$ denotes the Euclidean distance from $p$ to the origin.

\begin{proof}
Suppose not. Then there is a sequence of open, connected, area-minimizing surfaces
$M_i$ in $\h^3$ such that
\begin{enumerate}[(i)]
\item $\partial M_i$ is disjoint from $\{p\in\overline{\h^3}: r_i< |p| < R_i\}$, and
\item $M\cap\{|p|\le r_i\}$ and $M\cap \{p: |p|\ge R_i\}$ are both nonempty,
\item $R_i/r_i\to\infty$.
\end{enumerate}
Since dilations are hyperbolic isometries, we can assume that $R_i=1/r_i$,
so that $R_i\to\infty$ and $r_i\to 0$.  
By the Compactness Theorem~\ref{compactness-theorem},
a subsequence of the $M_i$ converges smoothly on compact subsets of $\h^3$
to a properly embedded minimal surface $M\subset \h^3$ such that $M$ intersects
each Euclidean sphere centered at the origin and such that the limit set $L(M)$ of $M$ is contained
in $\{0,\infty\}$. By the convex hull property (Proposition~\ref{convex-hull-proposition}),
$M$ is contained in the $z$-axis, which is  impossible.
\end{proof}

\begin{corollary}\label{far-apart-corollary}
Let $M_1\subset \h^3$ be a uniquely area-minimizing surface
and let $p$ be a point in $\overline{\h^3}\setminus \overline{M_1}$.
Then $p$ has a neighborhood $U\subset \overline{\h^3}$ with the following property:
if $M_2$ is a uniquely area-minimizing surfaces that lies in $U$, then $M_1\cup M_2$
is also uniquely area-minimizing.
\end{corollary}

\begin{proof}
We can work in the halfspace model of $\h^3$ with $p=(0,0,0)$.
Choose $R>0$ so that 
\[
  M_1\subset \{q: |q|\ge R\}.
\]
Let $U=\{q: |q|<r\}$, where $R/r$ is sufficiently
large that the conclusion of Lemma~\ref{far-apart-lemma} holds. 

Let $M_2$ be a uniquely area-minimizing surface in $U$, and
suppose that $M_1\cup M_2$ is not uniquely area-minimizing.  Then there is an area-minimizing
surface $M$ such that $\partial M=\partial M_1\cup \partial M_2$ and such that $M\ne M_1\cup M_2$.
Since $M_1$ and $M_2$ are each uniquely area-minimizing, there must be a connected component
of $M$ that contains points in $\{q: |q|\le r\}$ and points in $\{q: |q|\ge R\}$, contradicting 
Lemma~\ref{far-apart-lemma}.
\end{proof}

\begin{proposition} \label{whole}
Let $M$ be an open, connected, orientable surface. Then there exists a complete, proper, area-minimizing embedding
$f \colon M \rightarrow \h^3$ such that the limit set is $\partial\h^3$.
\end{proposition}
\begin{proof}
We want to modify the proof of Theorem~\ref{th:first} as follows: we construct a sequence 
$\{\Sigma'_n\}_{n \in \n}$ in such a way that it satisfies Properties (I$_n$), $\ldots$, (IV$_n$) (see page \pageref{properties}) and:
 \begin{quote}
 (V$_n$) {\em The Euclidean distance } from $\partial  \Sigma_n$ to
any point in $\partial\h^3$ is less that $1/n$.
\end{quote}
To do this, once we have obtained the minimal surface $\Sigma_n$ satisfying (I$_n$),$\ldots$,(IV$_n$), 
then
we proceed as follows: Let $\Omega_1$,$\ldots$, $\Omega_k$ be the connected components of $\partial\h^3 \setminus \partial  \Sigma_n$. Take one of these components, $\Omega_i$, $i \in \{1,\ldots,k\}$ and consider  a complete, totally geodesic disk $D_i$ in $\h^3$ satisfying:
\begin{itemize}
\item $D_i$ and $\Sigma_n$ are disjoint;
\item $\partial  D_i \subset \Omega_i$;
\item $\diam_{\R^3} (\partial  D_i) < \frac{1}{2n};$
\item $D_i \cup \Sigma_n$ is uniquely area minimizing and strictly $L^\infty$ stable.
\end{itemize}
Such a disk exists by Corollary~\ref{far-apart-corollary}.
Let $\Gamma_i$ be a smooth arc in $\Omega_i$ that connects $\partial \Sigma_n$ and $\partial D_i$ and
that is $\frac{1}{2n}$ close to every point in $\Omega_i$. Then, we apply Theorem~\ref{th:bridge} to construct a new surface by connecting $\Sigma_n$ with $D_i$ by a bridge along the arc $\Gamma_i$. Notice that the  surface obtained in this way has  the topology as $\Sigma_n$. We  call $\Sigma'_n$ the  surface obtained by repeating the above procedure for all $i \in \{1, \ldots , k\}$. If the width of the bridges is sufficiently small we can guarantee that $\Sigma_n$ satisfies (V$_n$). So, the limit surface $\Sigma$ would satisfy that its limit set
 $L(\Sigma)$ is $\partial\h^3$.
\end{proof}


\subsection{Regularity of the boundary}
Although the minimal embedding constructed in Theorem~\ref{th:first} is limit
of surfaces with smooth boundary, we cannot assert anything about the regularity
at infinity of the minimal surface that we have obtained. In the case of finite topology,
Oliveira and Soret \cite{oliveira-soret} constructed minimal embeddings that extends 
smoothly to $\overline{\h}^3$. Hence, we shall center our attention on the case of open
surfaces with infinite topology. If we do not care about the property that the limit sets  of different ends were disjoint,
then we can demonstrate the following:

\begin{theorem}
\label{th:smooth}
Let $S$ be  an open surface with infinite topology.  Then there exists a proper area-minimizing
 embedding of
$S$ into $\h^3$ such that the limit set  in $\partial\h^3$ is a smooth
curve except for one point. Moreover the area-minimizing embedding extends smoothly to an embedding of $S$ into $\overline{\h}^3$ except for that point.
\end{theorem}
\begin{proof} We will  the upper half-space model of $\h^3$,
so $\partial\h^3= \{z=0\} \cup \{\infty\}$.
Let $\mathcal{S}=\{ S_1 \subset S_2 \subset \cdots \subset S_n \subset \cdots \}$ a simple exhaustion
for the surface $S$. 
For $n \in \n$, we define $X_n=\{ (x,y,z) \in \overline{\h}^3 \; : \; 2(n-1)<x <2 n-1\}$ and $Y_n =\{ (x,y,z) \in \overline{\h}^3 \; : \; 2n-1 < x <2 n\}$.

Consider a totally geodesic disk $D_n$ contained in the region $X_n$ given by the semi-sphere centered at $(2n-3/2,0,0)$ and radius $r_n<1/2$. Let $A_n$ the minimal annulus obtained by adding a bridge to $D_n$ along a diameter of $\partial  D_n$.
Similarly, we can construct a minimal disk with a handle $T_n$, included in the region $Y_n$. First we add a bridge at infinity $B$ to a totally geodesic disk represented by  a semi-sphere centered at $(2n-1/2,0,0)$ and radius $r_n<1/2$. Later, we add a second bridge $B'$ along a curve in $\partial\h^3$ joining to opposites points of the ideal boundary of $B$. Notice that the surfaces
$A_n$ and $T_n$, $n \in \n$, satisfy the 
hypothesis of our bridge principle at infinity (Theorem~\ref{th:bridge}). 
\begin{figure}[htbp]
    \begin{center}
        \includegraphics[width=\textwidth]{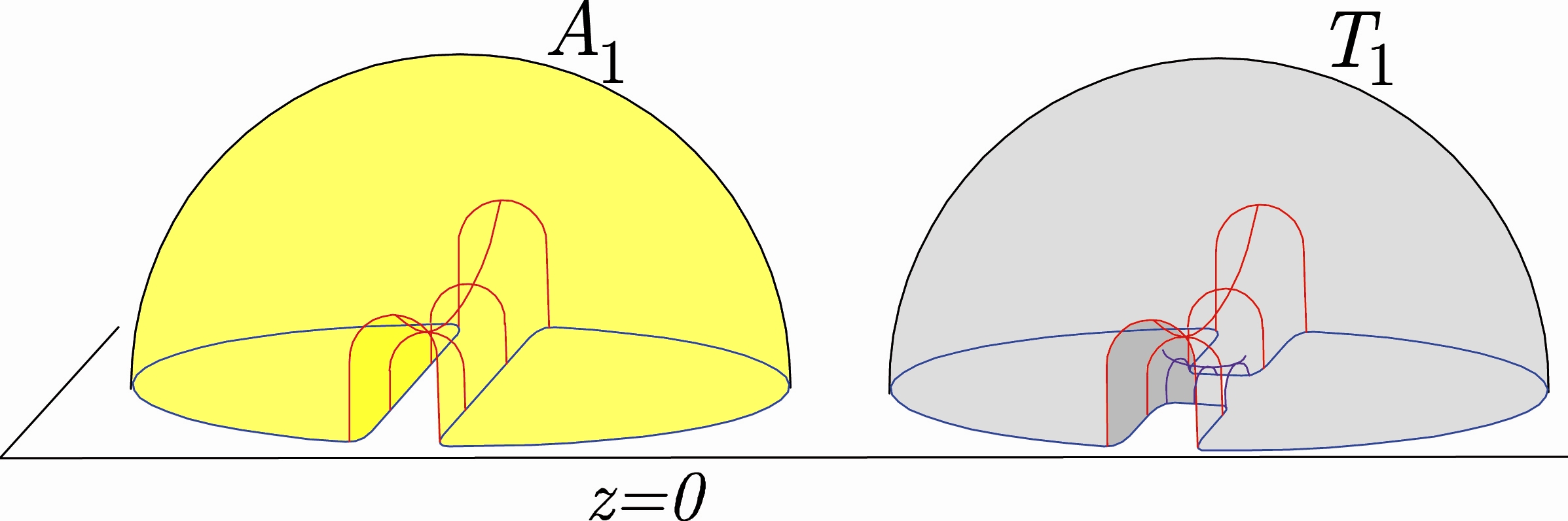}
    \end{center}
   \caption{The surfaces $A_1$ and $T_1$} \label{fig:ends-3}
\end{figure}

As in the proof of Theorem~\ref{th:first}, we construct our surface inductively.  The first element in our sequence is the totally geodesic disk $\Sigma_1=D_1$. The second element in the sequence, $\Sigma_2$,  is obtained by joining $\Sigma_1$ with $W_2 \in \{A_2,T_2\}$ by a bridge at infinity along a curve $\Gamma_2$ with is contained in $\partial\h^3 \cap \{x<4\}$. The choice of $W_2$ depends on the topology of $S_2 \setminus \mbox{\rm Int} (S_1)$.
To add this bridge,  we have to guarantee that  $\Sigma_1 \cup W_2$ satisfies the assumptions of Theorem~\ref{th:bridge}. It is clear that $\Sigma_1 \cup W_2$ is strictly $L^\infty$ stable, so we only need to check that it is {\em uniquely area-minimizing}. 
By corollary~\ref{far-apart-corollary},
this can be guaranteed by applying a suitable homothetical shrinking to $W_2$ with respect to $(5/2,0,0)$ or $(7/2,0,0)$ (depending on 
the nature of $W_2$). Observe that the ideal boundary $\partial \Sigma_2$ is a set of pairwise disjoint Jordan curves so that $\partial\h^3 \setminus \partial  \Sigma_2$ consists of a disjoint union of disks (actually, either one or two disks) and one unbounded connected component that is not simply connected and that we shall denote $\mathcal{C}_2$.

Assume that the surface $\Sigma_n$ is constructed in such a way that $\Sigma_n$ is 
diffeomorphic to $S_n$ and $\partial \Sigma_n$ consists of a finite set of pairwise disjoint Jordan curves and such that $\partial\h^3 \setminus \partial  \Sigma_n$ consists of a disjoint union of finitely many topological disks 
together with one unbounded component $\mathcal{C}_n$.
We are going to show how to construct the surface $\Sigma_{n+1}$. We know that $S_{n+1} \setminus \mbox{\rm Int}(S_n)$ contains exactly one non-annular connected component that we call $\Delta_{n+1}$. Let $\sigma_{n+1}\subset \partial \Sigma_n$ be the connected component of $\partial  \Sigma_n$ which corresponds to $\partial \Delta_{n+1} \cap 
\partial 
S_{n}$ and let $q_{n+1}=(x_{n+1}, y_{n+1},z_{n+1})$ be the point of $\sigma_{n+1}$ with the highest 
$x$-coordinate. We have that $x_{n+1} \in [m,m+1]$ for some $m \in \n$.

Then we are going to construct  a curve $\Gamma_{n+1} \subset \mathcal{C}_n \cap \{m \leq x < 2 (n+1) \}$ joining $q_{n+1}$ and $W_{n+1} \in \{A_{n+1}, T_{n+1}\}$, where $W_{n+1}$ depends on the topology of $\Delta_{n+1}$. To do this, we proceed as follows. The intersection of $\{(t,0,0) \; : \; t\geq x_{n+1} \}$ and $\overline{\mathcal{C}_n}$ consists of a finite
(disjoint) union of segments $\alpha_1 \cup \cdots \cup \alpha_l$ and a half-line $r$. Let $\alpha_{l+1}$ be the
piece of $r$ joining $\partial  \Sigma_n$ and $\partial  W_{n+1}$. For $j \in \{1, \ldots,l\}$, label $\beta_j$ the arc
in $\partial \Sigma_n$ that joins the end point of $\alpha_j$ and the initial  point of $\alpha_{j+1}$. Notice that, from
our method of construction, the $x$-coordinate is non-decreasing along 
   $\beta_j$, $j=1,2,\ldots,l$.  Let us define
$$\gamma=\alpha_1 \ast \beta_1 \ast \alpha_2 \ast \cdots \ast \alpha_l \ast \beta_l \ast \alpha_{l+1}.$$
The curve $\Gamma_{n+1}$ is a suitable perturbation of $\gamma$ satisfying that $\Gamma_{n+1} \subset \mathcal{C}_n \cap \{m \leq x < 2 (n+1) \}$, and that $\Gamma_{n+1}$
does not touch the $x$-axis.

Again, up to a suitable shrinking of $W_{n+1}$ we can assume that we are in the conditions for applying Theorem~\ref{th:bridge}, and so, we obtain $\Sigma_{n+1}$ by adding a bridge along $\Gamma_{n+1}$ to $\Sigma_n \cup W_{n+1}$. Observe that
the bridge can be chosen so that it does not intersect the $x$-axis.

It is important to notice that the sequence of surfaces $\{\Sigma_n\}_{n \in \n}$ constructed in this way satisfies that, for all $r>0$ , the ideal boundary $\partial  \Sigma_n$ intersects the region
$\{x \leq r\}$ in the same set of arcs, for $n$ sufficiently large.

It is important to note that 
for every $r$ and $n$, $(\partial  \Sigma_n)\cap \{x<r\}$ is a finite collection of arcs.
Furthermore, there is an $n$ such that
\begin{equation}\label{arcs-stabilize}
  (\partial \Sigma_n)\cap \{x<r\} = (\partial  \Sigma_k)\cap \{x<r\}
\end{equation}
for all $k\ge n$.

Reasoning as in the proof Theorem~\ref{th:first}, we can guarantee that the sequence $\{\Sigma_n\}_{n \in \n}$ converges smoothly on compact sets to a  properly embedded minimal surface $\Sigma$. 
From~\eqref{arcs-stabilize} we see that 
$\partial \Sigma \cap \{x \leq r\}=\partial  \Sigma_n \cap \{x \leq r\},$ 
for $n \in \n$ large enough. Thus $(\partial  \Sigma)\setminus \{\infty\}$ 
is smooth and properly embedded in $\partial\h^3\setminus\{\infty\}$. 
 \end{proof}


\end{document}